\documentclass[12pt]{amsart}
\usepackage{amscd}
\usepackage{amssymb}
\newcommand{\be}{\begin{equation}}
\newcommand{\ee}{\end{equation}}
\newcommand{\ba}{\begin{eqnarray}}
\newcommand{\ea}{\end{eqnarray}}
\newcommand{\ban}{\begin{eqnarray*}}
	\newcommand{\ean}{\end{eqnarray*}}

\def\XXint#1#2#3{{\setbox0=\hbox{$#1{#2#3}{\int}$}
		\vcenter{\hbox{$#2#3$}}\kern-.5\wd0}}

\newcommand{\Rk}{\noindent {\bf Remark} }

\newtheorem{theo}{Theorem}[section]

\begin{document}
	\newtheorem{defn}[theo]{Definition}
	\newtheorem{ques}[theo]{Question}
	\newtheorem{lem}[theo]{Lemma}
	\newtheorem{prop}[theo]{Proposition}
	\newtheorem{coro}[theo]{Corollary}
	\newtheorem{ex}[theo]{Example}
	\newtheorem{note}[theo]{Note}
	\newtheorem{conj}[theo]{Conjecture}
	\makeatletter
	\@namedef{subjclassname@2020}{%
		\textup{2020} Mathematics Subject Classification}
	\makeatother

	\title[5-dim Ricci Shrinker with CSC]{Rigidity of Five-dimensional Shrinking Gradient Ricci Solitons with Constant Scalar Curvature I}
	\author{Fengjiang Li}
	\address[Fengjiang Li]
	{Mathematical Science Research Center, Chongqing University of Technology, Chongqing 400054, China}
	\email{fengjiangli@cqut.edu.cn}
	
	\author{Jianyu Ou}
	\address[Jianyu Ou]
	{Department of Mathematics, Xiamen University, Xiamen 361005, China}
	\email{oujianyu@xmu.edu.cn}
	
	\author{Yuanyuan Qu}
	\address[Yuanyuan Qu]{School of Mathematical Sciences,Shanghai Key Laboratory of PMMP, East China Normal University, Shanghai 200241,
		China}
	\email{52285500012@stu.ecnu.edu.cn}
	
	\author{Guoqiang Wu}
	\address[Guoqiang Wu]
	{School of Science, Zhejiang Sci-Tech University, Hangzhou 310018, China}
	\email{gqwu@zstu.edu.cn}

	
	\date{}
	\maketitle

	\begin{abstract} We investigate five-dimensional shrinking gradient Ricci solion $(M^5, g,f)$  satisfying $Ric+\nabla^2f= \frac{1}{2} g$  with constant scalar curvature.
Fern\'{a}ndez-L\'{o}pez and Garc\'\i a-R\'\i o \cite{FR16} proved that the possible values of  scalar curvature are  $\{0,1,\frac{3}{2}, 2, \frac{5}{2}\}$. If $R=0$, the soliton is isometric to $\mathbb{R}^5$; if $R=2$, the soliton is a finite quotient of ${N}^{4}\times\mathbb{R}$ \cite{FR16}, where ${N}^4$ is a four-dimensional Einstein manifold;  if $R=\frac{5}{2}$, the soliton is Einstein \cite{Petersen-Wylie2}. The most difficult cases are $R=\frac{3}{2}$ or $R=1$.
In \cite{Li-Ou-Qu-Wu} we proved that it is a finite quotient of $\mathbb{R}^3\times \mathbb{S}^2$ if $R=1$.
In this paper, we prove that five-dimensional shrinking gradient Ricci soliton with $R=\frac{3}{2}$ is isometric to a finite quotient of $\mathbb{R}^2\times \mathbb{S}^3$.
	\end{abstract}
	
	\section{Introduction}
	Let  $(M, g)$ be an $n$-dimensional complete gradient Ricci soliton with the potential function $f$ satisfying
	\begin{align}\label{soliton}
	\text{Ric}+\nabla^2f=\lambda g
	\end{align}
	for some constant $\lambda$, where $\text{Ric}$ is the Ricci tensor of $g$ and $\nabla^2f$ denotes the Hessian of the potential function $f$.
	The Ricci soliton is said to be shrinking, steady, or expanding accordingly as $\lambda$ is positive, zero, or negative, respectively. By rescalling the metric $g$ by a positive constant, we can assume $\lambda\in \{\frac{1}{2}, 0, -\frac{1}{2}\}$.
	\smallskip	
	
	A gradient Ricci soliton is a self-similar solution to the Ricci flow which flows by diffeomorphism and homothety. The study of solitons has become increasingly important in both the study of the Ricci flow introduced by Hamilton \cite{Hamilton} and metric measure theory. Solitons play a direct role as singularity dilations in the Ricci flow proof of uniformization. In \cite{Perelman1}, Perelman introduced the ancient $\kappa$-solutions, which
	play an important role in the singularity analysis, and he also proved that suitable blow down limit of ancient $\kappa$-solutions must be a shrinking gradient Ricci soliton. In \cite{Perelman2}, Perelman proved that any two dimensional non-flat ancient $\kappa$-soluition must be the standard $\mathbb S^2$, and he also classified three-dimensional shrinking gradient Ricci soliton under the assumption of nonnegative curvature and $\kappa$-noncollapseness. Due to the work of Perelman \cite{Perelman2}, Ni-Wallach \cite{Ni-Wallach}, Cao-Chen-Zhu \cite{Cao-Chen-Zhu}, the classification of three-dimensional shrinking gradient Ricci soliton is complete. For more work on the classification of gradient Ricci soliton under various curvature condition, see \cite{Brendle1, Brendle2, Cao-Chen, Cao-Xie, Cao-Chen-Zhu, Cao-Chen2, Cao-Wang-Zhang, Chen-Wang,Kotschwar, Eminenti-LaNave-Mantegazza,Munteanu-Wang4,  Naber, Petersen-Wylie, Pigola-Rimoldi-Setti, Wu-Wu-Wylie, Wu-Zhang, Zhang}.

	\smallskip
	In general, it is hard to understand the geometry and topology of Ricci soliton in high dimensions, even in dimension four. Cao \cite{Cao} and Koiso \cite{Koiso} independently constructed a gradient K\"ahler Ricci soliton on $\mathbb{CP}^2\# (- \mathbb{CP}^2)$ which has $U(2)$ symmetry and $Ric>0$.
	Wang-Zhu \cite{Wang-Zhu} found a gradient K\"ahler Ricci soliton on $\mathbb{CP}^2 \# (-2 \mathbb{CP}^2)$ which has $U(1)\times U(1)$ symmetry.
	Feldman-Ilmanen-Knopf \cite{Feldman-Ilmanen-Knopf} constructed the first noncompact $U(2)$ invariant shrinking K\"ahler Ricci soliton on the tautological line bundle of $\mathbb{CP}^1$ (we call it $FIK$ soltion in the following) whose Ricci curvature changes sign.
	Recently, Bamler-Cifarelli-Conlon-Deruelle \cite{BCCD22} proved the existence of a unique complete shrinking gradient K\"ahler-Ricci soliton with bounded scalar curvature on the blowup of $\mathbb{C}\times \mathbb{CP}^{1}$ at one point, called BCCD soliton.
	Li-Wang  \cite{Li-Wang}  proved that any K\"ahler Ricci shrinker surface has bounded sectional curvature. Combining the work by Conlon-Deruelle-Sun \cite{Conlon-Deruelle-Sun} and Cifarelli-Conlon-Deruelle \cite{CCD22}, they provided a complete classification of all K\"ahler Ricci shrinker surfaces.
	
	\smallskip	
	In this paper, we focus our attention on five-dimensional gradient shrinking Ricci solitons with constant scalar curvature. Recall that in Petersen and Wylie's paper \cite{Petersen-Wylie}, a gradient Ricci soliton $(M, g)$ is said to be rigid if it is isometric to a quotient $N \times \mathbb{R}^k$, the product soliton of an Einstein manifold $N$ of positive scalar curvature with the Gaussian soliton $\mathbb{R}^k$.
	Conversely, for the complete shrinking case, Prof. Huai-Dong Cao raised the following
	
	\smallskip	
	\noindent {\bf Conjecture}:
	Let $(M^n, g, f)$, $n\geq 4$, be a complete $n$-dimensional gradient shrinking Ricci soliton. If $(M, g)$ has constant scalar curvature, then it must be rigid,
	i.e., a finite quotient of ${N}^k\times \mathbb{R}^{n-k}$ for some Einstein manifold ${N}$ of positive scalar curvature.
	
	\smallskip	
	
	Fern\'{a}ndez-L\'{o}pez-Garc\'\i a-R\'\i o \cite{FR10} and Munteanu-Sesum \cite{MS13} proved that $n$-dimensional complete gradient shrinking solitons with harmonic Weyl tensor are rigid. Catino-Mastrolia-Monticelli \cite{CMM17} showed that the gradient shrinking Ricci soliton with fourth-order divergence-free Weyl tensor is rigid.
	Petersen and Wylie \cite{Petersen-Wylie} proved that a complete gradient Ricci soliton is rigid if and only if it has
	constant scalar curvature and is radially flat, that is, the sectional curvature $K(\cdot, \nabla f)=0$.
	Particularly, Petersen and Wylie \cite{Petersen-Wylie} also showed that the scalar curvature $R$ of a gradient Ricci soliton is $0$ or $n\lambda$, if and only if the underlying Riemannian structure is Einstein.
	Subsequently, Fern\'{a}ndez-L\'{o}pez and Garc\'\i a-R\'\i o \cite{FR16} proved that
	a complete gradient Ricci soliton is rigid iff it has constant scalar curvature and the Ricci curvature has constant rank. They also derived the following results for complete $n$-dimensional gradient Ricci solitons \eqref{soliton} with constant scalar curvature $R$: (i) The possible value of $R$ is $\{0, \lambda, \cdots, (n-1)\lambda, n\lambda\}$. (ii) If $R$ takes the value $(n-1)\lambda$,  then the soliton must be rigid. (iii) In the shrinking case, there is  no any complete gradient shrinking Ricci soliton  with $R= \frac{1}{2}$.
	(iv) Any $n$-dimensional gradient shrinking Ricci soliton with constant scalar curvature $R= \frac{ n-2}{2}$ has non-negative Ricci curvature.

	\smallskip	
For four-dimensional complete shrinking gradient Ricci solitons with constant scalar curvature, by the results stated above, their scalar curvature  $R\in\{0, 1,  \frac{3}{2},  2\}$. Moreover, if $R=0$ or $ 2$, they are Einstein \cite{Petersen-Wylie}; If $R=\frac{3}{2}$, they are finite quotient of $\mathbb{S}^{3}\times\mathbb{R}$ \cite{FR16}.
	Very recently, Cheng and Zhou \cite{Cheng-Zhou} handled the most subtle case $R=1$. In this case, the Ricci curvature is non-negative by the work of Fern\'{a}ndez-L\'{o}pez and Garc\'\i a-R\'\i o \cite{FR16}.
	In addition, the Riemannian curvature is bounded since Munteanu-Wang  \cite{Munteanu-Wang3}  proved four-dimensional complete shrinking gradient Ricci solitons with bounded scalar curvature has bounded Riemannian curvature.
	However, for $n$-dimensional ($n\geq 5$), we cannot even guarantee the non-negativity of Ricci curvature and the boundedness of Riemannian curvature.
	Moreovre, we point out that even if the $n$-dimensional ($n\geq 5$) shrinking gradient Ricci soliton has non-negative Ricci curvature and bounded Riemannian curvature, Cao's \textbf{Conjecture} is still open.
	
	\smallskip	
	Now we focus on five-dimensional complete shrinking gradient Ricci solitons with constant scalar curvature. By the results in \cite{FR16} and \cite{Petersen-Wylie2} stated above,
	their scalar curvature $R\in \{0,  1,  \frac{3}{2}, 2, \frac{5}{2}\}$.  Moreover,
	there is no any complete gradient shrinking Ricci soliton  with $R=\lambda$; if $R=0$ or $\frac{5}{2}$, they are Einstein; if $R=2$,  they are finite quotient of $\mathbb{N}^{4}\times\mathbb{R}$, where $\mathbb{N}^4$ is a four-dimensional Einstein manifold.  Therefore, $R=1$ and $ \frac{3}{2} $ are unknown cases.
	
	\smallskip	
	In this paper, motivated by the four-dimension works \cite{Cheng-Zhou,Petersen-Wylie,FR16}, we study five-dimensional complete gradient Ricci solitons with constant scalar curvature $\frac{3}{2}$.
	Our main theorem is as follows.
	\begin{theo}\label{main}
		Suppose $(M^5, g, f)$ is a five-dimensional shrinking gradient Ricci soliton with $R= \frac{3}{2}$, then it is isometric to a finite quotient of $\mathbb{R}^2 \times \mathbb{S}^3$.
	\end{theo}

\textbf{Remark.}
Suppose $(M^5, g, f)$ is a five-dimensional shrinking gradient Ricci soliton with $R=1$,
this case is more subtle. We prove that it is isometric to a finite quotient of $\mathbb{R}^2 \times \mathbb{S}^3$ under the additional condition of bounded curvature \cite{Li-Ou-Qu-Wu}.

\medskip

	%
	Next, we discuss some methods to address this problem.
	Note that, as mentioned before, a complete gradient Ricci soliton is rigid if and only if it has constant scalar curvature and is radially flat \cite{Petersen-Wylie}. Fern-\\\'{a}ndez-L\'{o}pez and Garc\'\i a-R\'\i o  \cite{FR16}  proved the radial flatness is equivalent to the constant rank of Ricci curvature.
	Based on this, it is only necessary to prove that the Ricci curvature has constant rank for Cao's Conjecture.
	
	\smallskip	
In Cheng-Zhou's work \cite{Cheng-Zhou}, 	they applied the weighted Laplacian $\Delta_f$ to the quantity $\textrm{tr}(\text{Ric}^3)$, the trace of the tensor $\textrm{Ric}^3$, for four-dimensional gradient shrinking Ricci soliton with constant scalar curvature $2\lambda$ and then derived the following nice inequality
	\ban
	\Delta_f  \left[ f(\textrm{tr}(\text{Ric}^3)-\frac14)\right] \geq 9 f\left[ \textrm{tr}(\text{Ric}^3)-\frac14\right].
	\ean
	Using integration by parts, they concluded that $\textrm{tr}(\text{Ric}^3)-\frac14=0$
	over $M$, 
	implying that the Ricci curvature has rank 2, and thus they obtained the rigidity result.
	
	\smallskip	
	We would like to point out that $\frac{1}{3}\left[ \textrm{tr}(\text{Ric}^3)-\frac14\right]=\sigma_3(Ric)$,
	(see equation (3.11) in \cite{Cheng-Zhou},
	since 0 is a Ricci-eigenvalue of gradient shrinking Ricci soliton with constant scalar curvature.
	When the gradient shrinking Ricci soliton has nonnegative sectional curvature,
	Guan-Lu-Xu  \cite{GLX15}  used a combination of $\sigma_k(Ric)$ to prove that Ricci curvature has constant rank.
	Naber \cite{Naber} showed that any four-dimensional non-flat complete noncompact gradient shrinking Ricci soliton
	with the bounded non-negative curvature operator is also rigid.
	\smallskip
	
	Our new proof is inspired by the work of Petensen-Wylie \cite{Petersen-Wylie2}. If the sectional is curvature is nonnegative, it is easy to observe that $Rm*Ric\geq 0$, and they applied $\Delta_f$ directly to the sum of the smallest $k$ Ricci-eigenvalues, obtaining:
	\ban
	\Delta_f(\lambda_1+\lambda_2+ \cdot\cdot\cdot+ \lambda_k) \leq (\lambda_1+\lambda_2+\cdot\cdot\cdot+  \lambda_k)
	\ean
	in the barrier sense. Then they derived the desired result by the standard maximum principle. \\
	
	In the case of Cheng-Zhou  \cite{Cheng-Zhou}, suppose a four-dimensional gradient shrinking Ricci soliton has constant scalar curvature $2\lambda$, inspired by Pertersen-Wylie  \cite{Petersen-Wylie}, Ou-Qu-Wu  \cite{Ou-Qu-Wu}  applied $\Delta_f$ directly to the sum of the smallest $2$ Ricci-eigenvalues and showed it is isometric to a finite quotient of $\mathbb{S}^2\times \mathbb{R}^2$.
	
	\smallskip
	As noted in Cheng-Zhou's paper \cite{Cheng-Zhou}, in the proof of the nice inequality,
	the curvature tensor of $M$ is related with the curvature tensors of the level sets of $f$ by the Gauss equations.
	One crucial fact is that the level sets of the potential function $f$ is three-dimensional, and their curvature tensor can be expressed by its Ricci tensor, as their Weyl curvature is identically zero.
	Again, this crucial fact also plays an important role in Ou-Qu-Wu's calculations of $\Delta_f(\lambda_1+\lambda_2)$.
	However, difficulties arise in the five-dimensional case. Indeed, there is an additional term which involves Weyl curvature of the level set appears. In general, it is hard to control the Weyl curvature. Firstly, we use the four-dimensional Gauss-Bonnet-Chern formula to control the Weyl curvature, then the term $|\nabla Ric|$ occurs, we want to remark that there is a $f$ in the denominator, which enables us to  estimate the Weyl curvature successfully at last. Secondly, we can express $|\nabla Ric|^2$ in  terms of the Riemannian curvature and Ricci curvature due to the constant scalar curvature condition, and the most technical part is to estimate $|\nabla Ric|^2$ in an effective way, see Proposition \ref{key}. Thirdly, we can use the result obtained above to derive that the Riemannian curvature is bounded, which is essential for use to analyze the  asymptotic geometry at infinity, then we can get that $\lambda_1+\lambda_2\rightarrow 0$ at infinity. Finally we can use the integral argument to finish our Theorem.
	
	\smallskip

	The paper is organized as follows.
	In Section \ref{sec2}, we recall the notations and basic formulas on gradient shrinking Ricci solitons with constant scalar curvature.
	In Section \ref{sec3}, we will apply $\Delta_f$ directly to the sum of the smallest $2$ eigenvalues, denoted by $\lambda_1$ and $\lambda_2$ and then derive the estimate of $\Delta_f(\lambda_1+\lambda_2)$,  which involves the Weyl curvature of the level set as mentioned before, see Proposition \ref{lap12}.
	In Section \ref{sec4}, we prove a key estimate of $|\nabla Ric|^2$.
	In Section \ref{sec5}, based the point-picking argument, we prove the Riemannian curvature is bounded. By similar argument, we can also prove that $\lambda_1+\lambda_2 \rightarrow 0$ and $\nabla_{\nabla f}Ric$ also tends to zero at infinity.
	In Section \ref{sec6}, we prove Theorem \ref{main}.
	
	\smallskip

	
	\section{Notations and basic formulas on gradient shrinking Ricci solitons}\label{sec2}
	
	In this section, we recall the notations and basic formulas on gradient shrinking Ricci solitons with constant scalar curvature. For details, we refer to \cite{Cao,Hamilton,Petersen-Wylie,Cheng-Zhou}.
	
	\smallskip
	Let  $(M, g)$ be an $n$-dimensional complete gradient shrinking Ricci soliton satisfying \eqref{soliton}.
	By scaling the metric $g$,  one can normalize $\lambda$ so that $\lambda=\frac{1}{2}$. In this paper, we always assume $\lambda=\frac{1}{2}$ and the gradient shrinking Ricci soliton equation is as follows,
	\begin{align}\label{soliton'}
	\text{Ric}+\nabla^2f=\frac{1}{2} g.
	\end{align}
	
	At first we recall some basic formulas which will be used throughout the paper:
	\begin{align}\label{second bianchi}
	d R=2 Ric(\nabla f,\cdot),
	\end{align}
	\begin{align}\label{R}
	R+\Delta f=\frac{n}{2},
	\end{align}\begin{align}\label{f he tidu}
	R+|\nabla f|^2=f,
	\end{align}\begin{align}\label{lr}
	\Delta_f R=R-2|Ric|^2,
	\end{align}\begin{align}\label{elliptic equation}
	\Delta_f R_{ij}=R_{ij}-2 R_{ikjl}R_{kl}.
	\end{align}
	where $\Delta _{f}=\Delta -\left\langle \nabla f,\nabla \right\rangle $ is the weighted Laplacian, and $\Delta _{f}$ acting on the function
	is self-adjoint on the space of square integrable functions with respect to the weighted measure $e^{-f}dv.$ In general, the weighted Laplacian $\Delta_f$ acting on tensors is given by $\Delta_f=\Delta-\nabla_{\nabla f}$.
	
	\smallskip
	Next we state the key estimate of potential function $f$ in Cao-Zhou \cite{Cao-Zhou}.
	\begin{theo}[\cite{Cao-Zhou}]\label{potential estimate}  Suppose $(M^n, g, f)$ is an noncompact shrinking gradient Ricci soliton, then there exist $C_1$ and $C_2$ such that
		\ban\label{cao-zhou}
		\left(\frac{1}{2}d(x, p)-C_1\right)^2\leq f(x)\leq \left(\frac{1}{2}d(x, p)+C_2\right)^2,
		\ean
		where $p$ is the minimal point of $f$, which always exists.
	\end{theo}

	Now we consider complete gradient shrinking Ricci solitons with constant scalar curvature $R$. In this case, the potential function $f$ is isoparametric and the isoparametric property of plays a very important role.
	concretely, the potential function $f$ can be renormalized, by replacing $f-R$  with $f$, so that $f:M\to [0, +\infty)$ satisfies 	 \begin{equation}\label{iso1}
	|\nabla f|^2=f,
	\end{equation}
	which implies that $f$ is transnormal.
	Recall \eqref{R}
	\begin{align*}\label{R'}
	\Delta f=\frac{n}2-R,
	\end{align*}
	therefore the (nonconstant) renormalized $f$ is an isoparametric function on $M$. From the potential function estimate \eqref{cao-zhou},  $f$ is proper and unbounded.
	By the theory of isoparametric functions, Cheng-Zhou (\cite{Cheng-Zhou})
	derived the following results.
	\begin{theo}[\cite{Cheng-Zhou}]\label{levelset}
		Let $(M, g, f)$
		be an $n$-dimensional complete noncompact gradient shrinking Ricci soliton satisfying \eqref{soliton'} with constant scalar curvature $R$ and let $f$ be normalized as
		\[
		|\nabla f|^2=f.
		\]
		Then the following results hold.
		
		\smallskip	
		{\rm (i)} $M_{-}=f^{-1}(0)$ is a compact and  connected minimal submanifold of $M$.
		
		\smallskip	
		{\rm (ii)} The function $f$ can be expressed as
		\ban
		f(x)=\frac{1}{4} \text{dist}^2 (x, M_{-}).
		\ean
		
		\smallskip	
		{\rm (iii)} For any point $p\in M_{-}$, $\nabla^2f$ has two eigenspaces $T_pM_{-}$ and $\nu_pM_{-}$ corresponding eigenvalues $0$ and $\frac12$, and $\dim(M_{-})=2R$.\\
		
		\smallskip	
		{\rm (iv)} Let $D_a:=\{x\in M:\,f(x)\leq a\},$ for $a>0$. Then the mean curvature $H(a)$ of the smooth hypersurface $\Sigma(a)=\partial D_a$ satisfies
		\[H(a)=\frac{(n-2R-1)}{\sqrt{a}}.\]
		
		\smallskip	
		{\rm (v)} The volume of the set $D_a$ satisfies
		\[\textrm{Vol} (D_a)= \frac{2^k}{k}|M_{-}|\omega_k a^\frac{k}{2}, \, \textrm{Vol} (\Sigma_a)= 2^{k-1}|M_{-}|\omega_k a^\frac{k-1}{2},\]
		where $k=n-2R$, $|M_{-}|$ denotes the volume of the submanifold $M_{-}$, and  $\omega_{k-1}$ is the area of the unit sphere in $\mathbb{R}^k$.
	\end{theo}
	
	\smallskip

For later application,  we need to discuss the relation between barrier solution and distribution.

Let $(M,g)$ be a complete manifold, $\nabla$ be the gradient operator on $M$,
$\Delta$ be the Laplacian on $M$, $\Omega$ be an open set of $M$ and
$F\in C^\infty(M)$. The $F$-Laplacian is defined by
\[
\Delta_F:=\Delta-\nabla F\cdot\nabla.
\]
We say that a continuous function $u\in C(\Omega)$ satisfies $\Delta_F u\le w$
for some $w$ \emph{in the barrier sense}, if for any fixed $x\in \Omega$,
there exists a smooth function $v$ defined in a neighborhood $U(x)$ of $x$,
such that $u(x)=v(x)$, $u(y)\le v(y)$ for any $y\in U(x)$ and
\[
\Delta_F v(x)\le w(x).
\]
We say that $u\in C(\Omega)$ satisfies $\Delta_F u\le w$ on $\Omega$
\emph{in the sense of distribution}, if
\[
\int_{\Omega} u\Delta_F\phi\, e^{-F}dvol\leq \int_{\Omega} w \phi\, e^{-F}dvol
\]
for any $\phi\geq 0$ with $\phi\in C_c^\infty(\Omega)$.

In the appendix of \cite{Wu-Wu}, the authors proved the following Theorem.
\begin{theo}[\cite{Wu-Wu}]\label{bard}
If $u\in C(\Omega)$ satisfies $\Delta_F u\le w$ for some $w$
in the barrier sense, then it satisfies $\Delta_F u\le w$ on $\Omega$ in
the sense of distribution.
\end{theo}
Next we shall apply Theorem \ref{bard} to discuss some special cases for our
theorems in the preceding sections. We use the same notations as before. On
an $n$-dimensional shrinker $(M,g,f)$, let $u$ be an
 Lipschitz function and let $w$ be an integrable function
with respect to measure $e^{-h}dv$. Here $h$ is uniformly equivalent to the square of
distance function. For a sufficiently large
$r$, $a$ and $b$ ($a<b$), set $D(r):=\{x\in M|h(x)\le r\}$,
$\Sigma(r):=\{x\in M|h(x)=r\}$ and $D(a,b):=\{x\in M|a\le h(x)\le b\}$.
Then Theorem \ref{bard} implies that
\begin{coro}[\cite{Wu-Wu}]\label{bardis}
If $\Delta_h u\le w$ holds on $M\setminus D(r)$ in the barrier sense, then it
holds on $M\setminus D(r)$ in the sense of distribution.
\end{coro}
Using Corollary \ref{bardis}, we have the following useful proposition.
\begin{prop}[\cite{Wu-Wu}]\label{keyprop}
If a continuous and Lipschitz function $u$ satisfies $\Delta_h u\le w$
on $M\setminus D(r)$ in the barrier sense, then
\begin{align*}
-\int_{\Sigma(r)}\langle\nabla u, \tfrac{\nabla h}{|\nabla h|}\rangle e^{-h}dvol\le \int_{M\setminus D(r)} w\, e^{-h} d\sigma_{\Sigma(r)}
\end{align*}
holds for almost everywhere sufficiently large $r$.
\end{prop}
\Rk. During the paper,  $\int_{M\setminus D(a)} \Delta_h u\cdot e^{-h}dvol$ is understood as $-\int_{\Sigma(a)}\langle \nabla u, \frac{\nabla h}{|\nabla h|}\rangle e^{-h} d\sigma_{\Sigma(a)}$, for more details, see \cite{Wu-Wu}.

	
	%
	\section{Estimate on the sum of the smallest two Ricci-eigenvalues of weighted Laplacian operator}\label{sec3}
	In this section, let $(M, g, f)$ be a five-dimensional complete noncompact gradient shrinking Ricci soliton satisfying \eqref{soliton'} with constant scalar curvature $R=\frac{3}{2}$. We will apply $\Delta_f$ directly to the sum of the smallest $2$ Ricci-eigenvalues, denoted $\lambda_1$ and $\lambda_2$, and derive the estimate of $\Delta_f(\lambda_1+\lambda_2)$ involving the Weyl curvature of the level set, see Propositon \ref{lap12}.
	
	First, we recall the non-negativity of Ricci curvature for five-dimensional gradient shrinking Ricci soliton with constant scalar curvature $R=\frac{3}{2}$, also see \cite{FR16}.
	\begin{prop}[\cite{FR16}] \label{positive}
		Let $(M, g, f)$ be a five-dimensional complete noncompact gradient shrinking Ricci soliton satisfying \eqref{soliton'} with constant scalar curvature $R=\frac{3}{2}$. Then it has nonnegative Ricci curvature, the smallest Ricci-eigenvalue $\lambda_1=0$ and $Ric(\nabla f,\cdot)=0$.
	\end{prop}
	\begin{proof}
		Let $(M, g, f)$ be an $n$-dimensional complete noncompact gradient shrinking Ricci soliton satisfying \eqref{soliton'} with constant scalar curvature $R=\frac{n-2}{2}$, the authors proved that  the Ricci curvature is non-negative In \cite{FR16}, Thus for $n=5$ and $R=\frac{3}{2}$, it has nonnegative Ricci curvature.
		
		\smallskip
		\eqref{second bianchi} impiles $Ric(\nabla f,\cdot)=0$ since the scalar curvature is constant.
		Hence $\lambda_1=0$ is the smallest Ricci-eigenvalue with Ricci-eigenvector $\nabla f$.
		
	\end{proof}
	
	From Proposition \ref{positive}, throughout this paper we always denote the eigenvalues of Ricci curvature by
	\[
	0=\lambda_1\leq \lambda_2\leq  \lambda_3 \leq  \lambda_4\leq \lambda_5.
	\]
	
	\smallskip
	In the following, we calculate $\Delta_f(\lambda_1+\lambda_2)$ in the barrier sense, amd have the following lemma.
	\begin{lem} Let $(M^5, g, f)$ be a five-dimensional shrinking gradient Ricci soliton with constant scalar curvature $\frac{3}{2}$.  	
		Then
		\begin{equation}\label{l12}
		\Delta_f(\lambda_1+\lambda_2)\leq(\lambda_1+\lambda_2)-2\sum_{\alpha=2}^5 K_{1\alpha}\lambda_\alpha-2\sum_{\alpha=3}^5 K_{2\alpha}\lambda_\alpha
		\end{equation}	
		in the sense of barrier, where $K_{ij}$ denotes the sectional curvature of the plane spanned by $e_i$ and $e_j$, and $\{e_i\}_{i=1}^n$ is the orthonormal eigenvectors corresponding to the Ricci-eigenvalue $ \{\lambda_i\}_{i=1}^n$.
	\end{lem}
	
	\begin{proof}
		Actually, at $x$, because $R=\frac{3}{2}$, $Ric(\nabla f)=0$, so we choose $e_1=\frac{\nabla f}{|\nabla f|}$, then extend $e_1$ to an orthonormal basis $\{e_1, e_2, e_3, e_4,e_5\}$ such that $\{e_i\}_{i=1}^5$ are the eigenvectors of $Ric(x)$ corresponding to eigenvalues $\{\lambda_1, \lambda_2, \lambda_3, \lambda_4, \lambda_5\}$. Take parallel transport of $\{e_i\}_{i=1}^5$ along all the geodesics from $x$, then in a neighborhood $B(x, \delta)$ we get a smooth function $u(y)=Ric(y)(e_1(y), e_1(y))+Ric(y)(e_2(y), e_2(y))$ satisfying $u(y)\geq \lambda_1(y)+\lambda_2(y)$ and  $u(x)= \lambda_1(x)+\lambda_2(x)$.
		Thus, at $x$,
	 \begin{equation*}
		\begin{aligned}
		   \Delta_f u(x)=&\Delta_f\left(Ric(y)(e_1(y), e_1(y))+Ric(y)(e_2(y), e_2(y)) \right)|_{y=x}\\
			=&(\Delta_f Ric)(e_1, e_1)(x)+(\Delta_f Ric)(e_2, e_2)(x)\\
			=&(\lambda_1+\lambda_2)-2(\sum_{i=1}^5 K_{1i}\lambda_i +\sum_{i=1}^5 K_{2i}\lambda_i)\\
			=&(\lambda_1+\lambda_2)-2\sum_{\alpha=2}^5 K_{1\alpha}\lambda_\alpha
			-2\sum_{\alpha=3}^5 K_{2\alpha}\lambda_\alpha,
		\end{aligned}
	\end{equation*}
	this completes the proof.
\end{proof}
	
	Next, we deal with the term $\sum_{\alpha=2}^5 K_{1\alpha}\lambda_\alpha$. First we will give the following lemma for preparation.
	
	\begin{lem} Let $(M^5, g, f)$ be a five-dimensional shrinking gradient Ricci soliton with constant scalar curvature $\frac{3}{2}$. Then
		\begin{equation}\label{k1a}
		K_{1\alpha}=\frac{\nabla_{\nabla f}R_{\alpha\alpha} +\lambda_\alpha(\frac{1}{2}-\lambda_\alpha)}{f}
		\end{equation}	
		for $\alpha=2,  3 ,4, 5.$	
	\end{lem}
	
	\begin{proof}
		From the Ricci identity, we have
		\begin{equation*}
		\begin{aligned}
		&-R(\nabla f, e_\alpha, \nabla f, e_\beta) \\
		=&-\left( \nabla_\beta f_{\alpha k} - \nabla_kf_{\alpha \beta}\right) f_{k}\\
		=&\left( \nabla_\beta R_{\alpha k} - \nabla_kR_{\alpha \beta}\right) f_{k}\\
		=&-  \nabla_{\nabla f} R_{\alpha \beta}+\nabla_{\beta}(R_{\alpha k} f_k)- R_{\alpha k}f_{ k\beta}\\
		=&-  \nabla_{\nabla f} R_{\alpha \beta}- R_{\alpha k}\left( \frac{1}{2}g_{ k\beta}-R_{ k\beta} \right) \\
		=&- \nabla_{\nabla f} R_{\alpha \beta}-\left( \frac{1}{2}R_{\alpha \beta}-\sum_{k=1}^5R_{\alpha k}R_{k\beta}\right),
		\end{aligned}
		\end{equation*}
		where \eqref{second bianchi} was used in the third equality. Therefore, we see
		\begin{equation}\label{R1a1b}
		R(e_1, e_\alpha, e_1, e_\beta)=\frac{ \nabla_{\nabla f} R_{\alpha \beta}+\left( \frac{1}{2}R_{\alpha \beta}-\sum_{k=1}^5 R_{\alpha k}R_{k\beta}\right) }{f}
		\end{equation}
		due to $|\nabla f|^2=f$. \eqref{k1a} holds by setting $\beta=\alpha$ in \eqref{R1a1b}.
		This completes the proof of the lemma.
	\end{proof}
	
	\begin{lem}\label{lek1a}
		Let $(M^5, g, f)$ be a five-dimensional shrinking gradient Ricci soliton with constant scalar curvature $\frac{3}{2}$. Then we have
		\ban
		 -\sum_{\alpha=2}^5K_{1\alpha}\lambda_\alpha=-\frac{1}{f}\sum_{\alpha=2}^5\lambda_\alpha^2(\frac{1}{2}-\lambda_\alpha)=\frac{1}{f}\sum_{\alpha=2}^5(\lambda_\alpha-\frac{1}{2})^2\lambda_\alpha.
		\ean
	\end{lem}
	
	\begin{proof}
		First, since the scalar curvature is constant, \eqref{lr} implies
		\ban
		|Ric|^2=\frac{R}{2}=\frac{3}{4}.
		\ean
		This means $\sum_{\alpha=2}^5\lambda_\alpha=\lambda_2+\lambda_3+\lambda_4+\lambda_5=\frac{3}{2}$ and $\sum_{\alpha=2}^5\lambda_\alpha^2=\lambda_2^2+\lambda_3^2+\lambda_4^2+\lambda_5^2=\frac{3}{4}$. Hence
		\ban
		\sum_{\alpha=2}^5\lambda_\alpha(\frac{1}{2}-\lambda_\alpha)=0.
		\ean
		Recall (\ref{k1a}), and we obtain
		\ban
		&&-\sum_{\alpha=2}^5K_{1\alpha}\lambda_\alpha\\
		&=&-\frac{1}{f}\sum_{\alpha=2}^5
		\left[ \nabla f\cdot \nabla \lambda_\alpha +\lambda_\alpha(\frac{1}{2}-\lambda_\alpha)\right]\lambda_\alpha\\
		&=&-\frac{1}{f}\left[ \frac{1}{2} \nabla f\cdot \sum_{\alpha=2}^5\lambda_\alpha^2
		+\sum_{\alpha=2}^5\lambda_\alpha^2(\frac{1}{2}-\lambda_\alpha) \right] \\
		&=&-\frac{1}{f}\sum_{\alpha=2}^5\lambda_\alpha^2(\frac{1}{2}-\lambda_\alpha) \\
		&=&-\frac{1}{f}\sum_{\alpha=2}^5\left[- \lambda_\alpha(\frac{1}{2}-\lambda_\alpha)^2+\frac{1}{2}\lambda_\alpha(\frac{1}{2}-\lambda_\alpha) \right] \\
		&=&\frac{1}{f}\sum_{\alpha=2}^5 \lambda_\alpha(\frac{1}{2}-\lambda_\alpha)^2.
		\ean
		We have completed the proof of this proposition.
		
	\end{proof}

	Subsequently, we consider the level set $\Sigma$ of the potential function $f$ to handle the term $-2\sum_{\alpha=3}^5 K_{2\alpha}\lambda_\alpha$. For this purpose,
	recall that the intrinsic curvature tensor $R^{\Sigma}_{\alpha \beta \gamma \eta }$ and the extrinsic curvature tensor $R_{\alpha \beta \gamma \eta }$ of ${\Sigma}$ where $\{\alpha,\beta,\gamma,\eta\}\in\{2,3,4,5\}$, are related by the Gauss equation:
	\ban
	R^{\Sigma}_{\alpha \beta \gamma \eta }=R_{\alpha \beta \gamma \eta }+h_{\alpha \gamma}h_{\beta \eta}-h_{\alpha \eta }h_{\beta \gamma},
	\ean
	where $h_{\alpha \beta }$ denotes the components of the second fundamental form $A$ of ${\Sigma}$. Moreover,
	
	\smallskip
	\textbf{Claim}
	\ba\label{rsigma}
	R^{\Sigma}=R=\frac{3}{2};
	\ea
	\ba\label{ricsigma}
	Ric^{\Sigma}=Ric+\frac{\nabla_{\nabla f} Ric}{f};
	\ea
	\ba\label{ricsigma'}
	|Ric^{\Sigma}|^2=|Ric|^2+\frac{|\nabla_{\nabla f} Ric|^2}{f^2}.
	\ea	
	
	\begin{equation}\label{kab}
	\begin{aligned}
	K_{\alpha \beta }=&{\frac{1}{2}}(\lambda_{\alpha}+\lambda_{\beta})-{\frac{1}{4}}-{\frac{1}{2}}(K_{1\alpha }+K_{1\beta })+{\frac{1}{4f}}[1-(\lambda_{\alpha }+\lambda_{\beta})] \\
	&- {\frac{1}{2f}}\left[ ({\frac{1}{2}}-\lambda_{\alpha})^2+({\frac{1}{2}}-\lambda_{\beta})^2\right]-{\frac{1}{f}}({\frac{1}{2}}-\lambda_{\alpha})({\frac{1}{2}}-\lambda_{\beta})+W^{\Sigma}_{\alpha \beta },
	\end{aligned}
	\end{equation}
	where $W^{\Sigma}_{\alpha \beta}=W^{\Sigma}_{\alpha \beta\alpha \beta} $ denotes the components of the Weyl curvature of ${\Sigma}$.	\\

	In fact, it follows from the Gauss equation that
	\ban
	R^{\Sigma}_{\alpha  \beta }=R_{\alpha  \beta }-R_{1\alpha 1 \beta }+Hh_{\alpha  \beta }-h_{\alpha\gamma }h_{\gamma\beta}
	\ean
	and the scalar curvature $R^{\Sigma}$ of ${\Sigma}$ satisfies
	\ban
	R^{\Sigma}=R-2R_{11}+H^2-|A|^2.
	\ean
	Since $R=\frac{3}{2}$, $Ric(\nabla f, \cdot)=0$, $R_{1i}=0$, $i=1,...,5$, then
	\ban
	R^{\Sigma}=R+H^2-|A|^2.
	\ean
	Noting
	\ban
	h_{\alpha \beta }=\frac{f_{\alpha \beta }}{|\nabla f|}=\frac{\frac{1}{2}-\lambda_\alpha}{\sqrt{f}}\delta_{\alpha \beta},
	\ean
	then the mean curvature satisfies
	\[
	H=\frac{\frac{4}{2}-\sum\lambda_\alpha}{\sqrt{f}}=\frac{1}{2\sqrt{f}}
	\]
	and
	\begin{align*}
	|A|^2=&\frac{1}{f}\sum(\frac{1}{2}-\lambda_\alpha)^2=\frac{1}{f}(1-\sum\lambda_\alpha+\sum \lambda_\alpha^2)\\
	=&\frac{1}{f}(1-\frac{3}{2}+\frac{3}{4})=\frac{1}{4f}.
	\end{align*}
	Hence, $R^{\Sigma}=R=\frac{3}{2}$. Together with \eqref{R1a1b}, we see
	\begin{equation*}
	\begin{aligned}
	R^{\Sigma}_{\alpha  \beta }=&R_{\alpha  \beta }-\frac{\nabla_{\nabla f} R_{\alpha \beta}+\left( \frac{1}{2}R_{\alpha \beta}-R_{\alpha k}R_{k\beta}\right) }{f}
	+\frac{\frac{1}{2}-\lambda_\alpha}{2f}\delta_{\alpha \beta}\\
	&-\frac{(\frac{1}{2}-\lambda_\alpha)(\frac{1}{2}-\lambda_\beta)}{f}\delta_{\alpha \gamma}\delta_{\gamma\beta}\\
	=&R_{\alpha  \beta }-\frac{ \nabla_{\nabla f} R_{\alpha \beta}}{f}+\frac{1}{f}[-\lambda_\alpha(\frac{1}{2}-\lambda_\alpha)+\frac{1}{2}(\frac{1}{2}-\lambda_\alpha)-(\frac{1}{2}-\lambda_\alpha)^2]\delta_{\alpha \beta}\\
	=&R_{\alpha  \beta }-\frac{\nabla _{\nabla f}  R_{\alpha \beta}}{f},
	\end{aligned}
	\end{equation*}
	which implies \eqref{ricsigma} holds.
	
	\smallskip
	By $|Ric|^2=\frac{3}{4}$ agian, we have $Ric\cdot\nabla_{\nabla f}Ric=0$.
	\ban
	|Ric^{\Sigma}|^2&&=|Ric|^2+\frac{|\nabla_{\nabla f} Ric|^2}{f^2}-2\frac{Ric\cdot\nabla_{\nabla f}Ric}{f}\\
	&&=|Ric|^2+\frac{|\nabla_{\nabla f} Ric|^2}{f^2}.
	\ean
	
	Finally, recall that the relationship between curvature and the Weyl curvature
	\ban
	R_{ijkl}=&&W_{ijkl}+{\frac{1}{n-2}}(g_{ik}R_{jl}-g_{il}R_{jk}-g_{jk}R_{il}+g_{jl}R_{ik})\\
	&&-{\frac{1}{(n-1)(n-2)}}R(g_{ik}g_{jl}-g_{il}g_{jk}),
	\ean
	and we get
	\begin{equation*}
	\begin{aligned}
	K^{\Sigma}_{\alpha \beta }=&W^{\Sigma}_{\alpha \beta }+{\frac{1}{2}}(R^{\Sigma}_{\alpha \alpha }+R^{\Sigma}_{\beta \beta})-{\frac{1}{6}}R^{\Sigma}\\
	=&W^{\Sigma}_{\alpha \beta }+{\frac{1}{2}}(R_{\alpha \alpha}-K_{1\alpha }+Hh_{\alpha \alpha}-h_{\alpha \alpha}^2+R_{\beta \beta}-K_{1\beta}+Hh_{\beta \beta }-h_{\beta \beta}^2)-{\frac{1}{4}}\\
	=&W^{\Sigma}_{\alpha \beta }+{\frac{1}{2}}\left(\lambda_{\alpha}+\lambda_{\beta}-K_{1\alpha }-K_{1\beta  }+H(h_{\alpha \alpha}+h_{\beta \beta })-h_{\alpha \alpha}^2-h_{\beta \beta}^2\right)-{\frac{1}{4}}
	\end{aligned}
	\end{equation*}
	on the  hypersurface $\Sigma$. It follows from the Gauss equation that
	\begin{equation*}
	\begin{aligned}
	K_{\alpha \beta }=&K^{\Sigma}_{\alpha \beta }-h_{\alpha \alpha}h_{\beta \beta}+h_{\alpha \beta }^2\\
	=&{\frac{1}{2}}\left(\lambda_{\alpha}+\lambda_{\beta}-K_{1\alpha }-K_{1\beta }+H(h_{\alpha \alpha}+h_{\beta \beta })-h_{\alpha \alpha}^2-h_{\beta \beta}^2\right)\\
	&-{\frac{1}{4}}-h_{\alpha \alpha}h_{\beta \beta}+h_{\alpha\beta}^2+W^{\Sigma}_{\alpha \beta }\\
	=&{\frac{1}{2}}(\lambda_{\alpha}+\lambda_{\beta})-{\frac{1}{4}}-{\frac{1}{2}}(K_{1\alpha }+K_{1\beta  })+{\frac{1}{4f}}[({\frac{1}{2}}-\lambda_{\alpha })+({\frac{1}{2}}-\lambda_{\beta})]\\
	&- {\frac{1}{2f}}\left[ ({\frac{1}{2}}-\lambda_{\alpha})^2+({\frac{1}{2}}-\lambda_{\beta})^2\right]  -{\frac{1}{f}}({\frac{1}{2}}-\lambda_{\alpha})({\frac{1}{2}}-\lambda_{\beta})\\
&+W^{\Sigma}_{\alpha \beta }.
	\end{aligned}
	\end{equation*}
	We have completed the proof of equations \eqref{rsigma}-\eqref{kab} in \textbf{Claim}.\\

	Using equations \eqref{rsigma}-\eqref{kab}, we have the following result.
	
	\begin{lem}\label{lek2a}
		Let $(M^5, g, f)$ be a five-dimensional shrinking gradient Ricci soliton with constant scalar curvature $\frac{3}{2}$. Then
 \begin{equation*}
	\begin{aligned}
		&-2\sum_{\alpha=3}^5K_{2\alpha}\lambda_\alpha\\
		=&-2(\lambda_1+\lambda_2)[1-(\lambda_1+\lambda_2)] +\frac{3}{2f}\nabla f\cdot \nabla (\lambda_1+\lambda_2)\\
		&-\frac{1}{f}\nabla f\cdot \nabla(\lambda_1+\lambda_2)^2
		-\frac{2}{f}\lambda_2({\frac{1}{2}}-\lambda_{2})^2-2\sum_{\alpha=3}^5 W^{\Sigma}_{2\alpha}\lambda_{\alpha}.
	\end{aligned}
\end{equation*}
\end{lem}

\begin{proof}
It follows from \eqref{kab} that
\begin{equation*}
	\begin{aligned}
		&-2\sum_{\alpha=3}^5K_{2\alpha}\lambda_\alpha\\
		 =&-\sum_{\alpha=3}^5(\lambda_{2}+\lambda_{\alpha})\lambda_\alpha+\frac{1}{2}\sum_{\alpha=3}^5\lambda_\alpha+\sum_{\alpha=3}^5(K_{12}+K_{1\alpha })\lambda_\alpha\\
		&-{\frac{1}{2f}}\sum_{\alpha=3}^5\left[ ({\frac{1}{2}}-\lambda_{2})+({\frac{1}{2}}-\lambda_{\alpha})\right] \lambda_\alpha\\
		&+\frac{1}{f}\left[({\frac{1}{2}}-\lambda_{2})^2\sum_{\alpha=3}^5\lambda_{\alpha}+ \sum_{\alpha=3}^5({\frac{1}{2}}-\lambda_{\alpha})^2 \lambda_{\alpha}\right]\\
		&+\frac{2}{f}({\frac{1}{2}}-\lambda_{2})\sum_{\alpha=3}^5({\frac{1}{2}}-\lambda_{\alpha})\lambda_{\alpha}
		-2\sum_{\alpha=3}^5W^{\Sigma}_{2\alpha}\lambda_{\alpha}\\
	\end{aligned}
\end{equation*}

Next, we handle these items one by one. First, notice that $\sum_{\alpha=3}^5\lambda_\alpha=\frac{3}{2}-\lambda_2$ and $\sum_{\alpha=3}^5\lambda_\alpha^2=\frac{3}{4}-\lambda_2^2$, it is easy to check
\begin{equation*}
	\begin{aligned}
		&-\sum_{\alpha=3}^5(\lambda_{2}+\lambda_{\alpha})\lambda_\alpha+\frac{1}{2}\sum_{\alpha=3}^5\lambda_\alpha\\
		=&-\lambda_{2}(\frac{3}{2}-\lambda_{2})-(\frac{3}{4}-\lambda_{2}^2)+\frac{1}{2}(\frac{3}{2}-\lambda_{2})\\
		=&-2\lambda_{2}(1-\lambda_{2}).
	\end{aligned}
\end{equation*}

It follows from \eqref{k1a} that
\begin{equation*}
	\begin{aligned}
		&\sum_{\alpha=3}^5(K_{12}+K_{1\alpha })\lambda_\alpha\\
		=& K_{12}({\frac{3}{2}}-2\lambda_{2})+\sum_{\alpha=2}^5K_{1\alpha}\lambda_\alpha\\
		=&({\frac{3}{2}}-2\lambda_{2})\frac{1}{f}\left[\nabla f\cdot \nabla \lambda_2 +\lambda_2(\frac{1}{2}-\lambda_2) \right]+\sum_{\alpha=2}^5K_{1\alpha}\lambda_\alpha\\
		=&\frac{3}{2f}\nabla f\cdot \nabla \lambda_2-\frac{1}{f}\nabla f\cdot \nabla \lambda_2^2
		+\frac{1}{f}\lambda_2(\frac{1}{2}-\lambda_2)({\frac{3}{2}}-2\lambda_{2})
		+\sum_{\alpha=2}^5K_{1\alpha}\lambda_\alpha
	\end{aligned}
\end{equation*}
By use of the fact that $\sum_{\alpha=3}^5\lambda_\alpha=\frac{3}{2}-\lambda_2$ and $\sum_{\alpha=3}^5\lambda_\alpha^2=\frac{3}{4}-\lambda_2^2$ agian, we see
\begin{equation*}
	\begin{aligned}
		&-{\frac{1}{2f}}\sum_{\alpha=3}^5\left[ ({\frac{1}{2}}-\lambda_{2})+({\frac{1}{2}}-\lambda_{\alpha})\right] \lambda_\alpha\\
		 &+\frac{1}{f}(\frac{1}{2}-\lambda_{2})^2\sum_{\alpha=3}^5\lambda_{\alpha}+\frac{2}{f}({\frac{1}{2}}-\lambda_{2})\sum_{\alpha=3}^5({\frac{1}{2}}-\lambda_{\alpha})\lambda_{\alpha}\\
		=&-\frac{1}{f}\lambda_2\left( 3\lambda^2_2-\frac{7}{2}\lambda_2+1\right)
	\end{aligned}
\end{equation*}

Therefore, we have
\begin{equation*}
	\begin{aligned}
		&-2\sum_{\alpha=3}^5K_{2\alpha}\lambda_\alpha\\
		=&-2\lambda_{2}(1-\lambda_{2})+\frac{3}{2f}\nabla f\cdot \nabla \lambda_2-\frac{1}{f}\nabla f\cdot \nabla \lambda_2^2+\sum_{\alpha=2}^5K_{1\alpha}\lambda_\alpha\\
		&+\frac{1}{f}\sum_{\alpha=3}^5({\frac{1}{2}}-\lambda_{\alpha})^2\lambda_{\alpha}
		-\frac{1}{f}\lambda_2\left[ 3\lambda^2_2-\frac{7}{2}\lambda_2+1 -(\frac{1}{2}-\lambda_2)({\frac{3}{2}}-2\lambda_{2})\right] \\
		&-2\sum_{\alpha=3}^5 W^{\Sigma}_{2\alpha}\lambda_{\alpha}.
	\end{aligned}
\end{equation*}
Since,
\begin{equation*}
	\begin{aligned}
		&\frac{1}{f}\sum_{\alpha=3}^5({\frac{1}{2}}-\lambda_{\alpha})^2\lambda_{\alpha}
		-\frac{1}{f}\lambda_2\left[ 3\lambda^2_2-\frac{7}{2}\lambda_2+1 -(\frac{1}{2}-\lambda_2)({\frac{3}{2}}-2\lambda_{2})\right] \\
		=&\frac{1}{f}\sum_{\alpha=3}^5({\frac{1}{2}}-\lambda_{\alpha})^2\lambda_{\alpha}
		-\frac{1}{f}\lambda_2({\frac{1}{2}}-\lambda_{2})^2 \\
		=&\frac{1}{f}\sum_{\alpha=2}^5({\frac{1}{2}}-\lambda_{\alpha})^2\lambda_{\alpha}
		-\frac{2}{f}\lambda_2({\frac{1}{2}}-\lambda_{2})^2,
	\end{aligned}
\end{equation*}
it follows from Lemma \ref{lek1a} that
\begin{equation*}
	\begin{aligned}
		&-2\sum_{\alpha=3}^5K_{2\alpha}\lambda_\alpha\\
		=&-2\lambda_{2}(1-\lambda_{2})+\frac{3}{2f}\nabla f\cdot \nabla \lambda_2-\frac{1}{f}\nabla f\cdot \nabla \lambda_2^2\\
		&-\frac{2}{f}\lambda_2({\frac{1}{2}}-\lambda_{2})^2-2\sum_{\alpha=3}^5 W^{\Sigma}_{2\alpha}\lambda_{\alpha}.
	\end{aligned}
\end{equation*}
The proof of Lemma \ref{lek2a} was thus completed.
\end{proof}
	
	We have the following two basic inequalities to obtain our desired estimate.
	\begin{lem}
		Let $(M^5, g, f)$ be a five-dimensional shrinking gradient Ricci soliton with constant scalar curvature $\frac{3}{2}$. Then
		\begin{equation} \label{la}
		(\frac{1}{2}-\lambda_\alpha)^2\leq\sum_{\alpha=3}^5(\frac{1}{2}-\lambda_\alpha)^2=\lambda_2(1-\lambda_2)
		\end{equation}
		for $\alpha=3,4,5$;
		\begin{equation}\label{l345}
		\sum_{\alpha=3}^5(\frac{1}{2}-\lambda_\alpha)^2\lambda_\alpha\leq\lambda_2(1-\lambda_2)(\frac{3}{2}-\lambda_2);
		\end{equation}
		\begin{equation}\label{l345'}
		\sum_{\alpha=3}^5(\frac{1}{2}-\lambda_\alpha)^2\lambda_\alpha^2\leq\lambda_2(1-\lambda_2)(\frac{3}{4}-\lambda_2^2).
		\end{equation}
	\end{lem}

	\begin{proof}
		It follows from \eqref{lr} that
		\begin{equation*}
		\begin{aligned}
		0=&\sum_{\alpha=2}^5\lambda_{\alpha}({\frac{1}{2}}-\lambda_{\alpha})\\
		=&-\sum_{\alpha=2}^5\left[ ({\frac{1}{2}}-\lambda_{\alpha})^2+\frac{1}{2}({\frac{1}{2}}-\lambda_{\alpha})\right] \\
		=&-\sum_{\alpha=2}^5({\frac{1}{2}}-\lambda_\alpha)^2+\frac{1}{4},
		\end{aligned}
		\end{equation*}
		which yeilds that
		\begin{equation*}
		\begin{aligned}
		\sum_{\alpha=3}^5({\frac{1}{2}}-\lambda_{\alpha})^2
		=&\frac{1}{4}-({\frac{1}{2}}-\lambda_2)^2\\
		=&\lambda_2(1-\lambda_2).		
		\end{aligned}
		\end{equation*}
		Therefore, we have
		\begin{equation*}
		\begin{aligned}
		&\,\sum_{\alpha=3}^5(\frac{1}{2}-\lambda_\alpha)^2\lambda_\alpha\\
		\leq&\, \lambda_2(1-\lambda_2) (\lambda_3+\lambda_4+\lambda_5)\\
		=&\,\lambda_2(1-\lambda_2)(\frac{3}{2}-\lambda_2).		
		\end{aligned}
		\end{equation*}
		Similarly, \eqref{l345'} holds.
		
	\end{proof}
	
	Finally, we have the following proposition.
	
	\begin{prop}\label{lap12}
		Let $(M^5, g, f)$ be a five-dimensional shrinking gradient Ricci soliton with constant scalar curvature $R=\frac{3}{2}$.  We have the following inequality holds in the barrier sense,
		\begin{equation*}
		\begin{aligned}
		&\Delta_f(\lambda_1+\lambda_2)\\
		\leq&-(\lambda_1+\lambda_2)+2(\lambda_1+\lambda_2)^2+\frac{3}{2f}\nabla f\cdot \nabla (\lambda_1+\lambda_2)\\
		&-\frac{1}{f}\nabla f\cdot \nabla (\lambda_1+\lambda_2)^2
		+\frac{2}{f}(\lambda_1+\lambda_2)[1-(\lambda_1+\lambda_2)]\left[ \frac{3}{2}-(\lambda_1+\lambda_2)\right]\\
		&-2\sum_{\alpha=3}^5 W^{\Sigma}_{2\alpha}\lambda_{\alpha}
		\end{aligned}
		\end{equation*}	
		on $M\setminus D(a)$ for some $a>0$.	
	\end{prop}
	
	\begin{proof} By inserting equations from Lemma \ref{lek2a} and \ref{lek1a} into \eqref{l12}, we obtain:
		\begin{equation*}
		\begin{aligned}
		&\Delta_f(\lambda_1+\lambda_2)\\
		\leq&\lambda_2-2\sum_{\alpha=2}^5K_{1\alpha}\lambda_\alpha-2\sum_{\alpha=3}^5K_{2\alpha}\lambda_\alpha\\
		=&-\lambda_2+2\lambda_2^2+\frac{3}{2f}\nabla f\cdot \nabla \lambda_2-\frac{1}{f}\nabla f\cdot \nabla \lambda_2^2\\
		&+\frac{2}{f}\sum_{\alpha=3}^5({\frac{1}{2}}-\lambda_{\alpha})^2\lambda_{\alpha}
		-2\sum_{\alpha=3}^5 W^{\Sigma}_{2\alpha}\lambda_{\alpha}\\
		\leq&-\lambda_2+2\lambda_2^2+\frac{3}{2f}\nabla f\cdot \nabla \lambda_2-\frac{1}{f}\nabla f\cdot \nabla \lambda_2^2\\
		&+\frac{2}{f}\lambda_2(1-\lambda_2)(\frac{3}{2}-\lambda_2)
		-2\sum_{\alpha=3}^5 W^{\Sigma}_{2\alpha}\lambda_{\alpha},
		\end{aligned}
		\end{equation*}	
		where \eqref{l345} was used in the last inequality, and the proof of this proposition was completed.
	\end{proof}
	
	\textbf{Remark}. In Proposition \ref{lap12}, the estimate of $\Delta_f(\lambda_1+\lambda_2)$ involves the term related to the Weyl curvature of the level set, which does not appear in the case of four-dimensional gradient Ricci solitons, thereby making the higher-dimensional case more difficult.

	\section{A key estimate of $|\nabla Ric|^2$}\label{sec4}
	In this section, we will derive a key estimate of $|\nabla Ric|^2$,
	focusing on effective control of the Weyl curvature.
	\smallskip
	
	Next, we proceed to calculate $|\nabla Ric|^2$.
	\begin{lem}\label{ricci}
		Let $(M^5, g, f)$ be a five-dimensional shrinking gradient Ricci soliton with constant scalar curvature $R=\frac{3}{2}$. Then we have
		\begin{equation*}
		|\nabla Ric|^2=2\sum_{\alpha,\beta=2,\alpha\neq\beta}^5 K_{\alpha\beta}\lambda_{\alpha}\lambda_{\beta}- \frac{3}{4}
		\end{equation*}	
	\end{lem}
	\begin{proof}
		Notice that $|Ric|^2=\frac{1}{2}R=\frac{3}{4}$ agian, we start with
		\ban
		\frac{1}{2}\Delta_f(|Ric|^2)=R_{ij}\Delta_fR_{ij}+|\nabla Ric|^2,
		\ean
		which implies
		\begin{equation*}
		\begin{aligned}
		|\nabla Ric|^2&=-\sum_{i,j=1}^5R_{ij}\Delta_fR_{ij}\\
		&=\sum_{i,j=1}^5R_{ij}(2R_{ikjl}R_{kl}-R_{ij})\\
		&=2\sum_{\alpha,\beta=2,\alpha\neq\beta}^5R_{\alpha\beta\alpha\beta}\lambda_\alpha\lambda_\beta-|Ric|^2\\
		&=2\sum_{\alpha,\beta=2,\alpha\neq\beta}^5K_{\alpha\beta}\lambda_{\alpha}\lambda_{\beta}- \frac{3}{4}.
		\end{aligned}
		\end{equation*}
		This completes the proof of Lemma \ref{ricci}.
	\end{proof}
	
	Subsequently, we deal with the right hand side of the equation in Lemma \ref{ricci}, and have the following key estimate of $|\nabla Ric|^2$.
	
	\begin{prop}\label{key}
		Let $(M^5, g, f)$ be a five-dimensional shrinking gradient Ricci soliton with constant scalar curvature $R=\frac{3}{2}$. There exists a constant $C>0$, such that
		\[
		|\nabla Ric|^2\leq C(\lambda_1+\lambda_2)+K_{12}+C|W^{\Sigma(s)}|^2	
		\]
		on $M\setminus D(a)$ for some $a>0$ and $s\geq a$.
	\end{prop}	

	\begin{proof}
		It follows from \eqref{kab} that
		\ban
		&&|\nabla Ric|^2=\sum_{\alpha,\beta=2,\alpha\neq\beta}^5 2K_{\alpha\beta}\lambda_{\alpha}\lambda_{\beta}-\frac{3}{4}\\
		 &=&\left(\sum_{\alpha,\beta=2,\alpha\neq\beta}^5(\lambda_{\alpha}+\lambda_{\beta})-{\frac{1}{2}}\right)\lambda_{\alpha}\lambda_{\beta}-\frac{3}{4}\\
		&&-{\sum_{\alpha,\beta=2,\alpha\neq\beta}^5}( K_{1\alpha}+K_{1\beta})\lambda_\alpha\lambda_\beta\\
		&&+\sum_{\alpha,\beta=2,\alpha\neq\beta}^5{\frac{1}{2f}}[({\frac{1}{2}}-\lambda_{\alpha })+({\frac{1}{2}}-\lambda_{\beta})]\lambda_\alpha\lambda_\beta\\
		&&-\frac{1}{f}\sum_{\alpha,\beta=2,\alpha\neq\beta}^5\left[ ({\frac{1}{2}}-\lambda_{\alpha})^2+({\frac{1}{2}}-\lambda_{\beta})^2\right] \lambda_\alpha\lambda_\beta\\
		 &&-\frac{2}{f}\sum_{\alpha,\beta=2,\alpha\neq\beta}^5({\frac{1}{2}}-\lambda_{\alpha})({\frac{1}{2}}-\lambda_{\beta})\lambda_\alpha\lambda_\beta\\
&&+2\sum_{\alpha,\beta=2,\alpha\neq\beta}^5W^{\Sigma}_{\alpha\beta}\lambda_\alpha\lambda_\beta.
		\ean
		Next, we handle these terms one by one.
		
		\textbf{Claim 1}.
		\ban
		 I&:=&\left(\sum_{\alpha,\beta=2,\alpha\neq\beta}^5(\lambda_{\alpha}+\lambda_{\beta})-{\frac{1}{2}}\right)\lambda_{\alpha}\lambda_{\beta}-\frac{3}{4}\\
		&\leq&-\frac{3}{2}(\lambda_1+\lambda_2)+6(\lambda_1+\lambda_2)^2-6(\lambda_1+\lambda_2)^3+6(\lambda_1+\lambda_2)^{\frac{3}{2}}.
		\ean
		In fact, we rewrite $I$ as follows,
		\begin{equation}\label{i}
		\begin{aligned}
		 I=&2\sum_{\alpha,\beta=2,\alpha\neq\beta}^5\lambda_{\alpha}^2\lambda_{\beta}-\frac{1}{2}\sum_{\alpha,\beta=2,\alpha\neq\beta}^5\lambda_{\alpha}\lambda_{\beta}-\frac{3}{4}\\
		 =&2\sum_{\alpha=2}^5\lambda_{\alpha}^2(\frac{3}{2}-\lambda_\alpha)-\frac{1}{2}\sum_{\alpha=2}^5\lambda_{\alpha}(\frac{3}{2}-\lambda_\alpha)-\frac{3}{4}\\
		 =&-2\sum_{\alpha=2}^5\lambda_{\alpha}^3+\frac{7}{2}\sum_{\alpha=2}^5\lambda_{\alpha}^2-\frac{3}{4}\sum_{\alpha=2}^5\lambda_{\alpha}-\frac{3}{4}\\
		=&-2\sum_{\alpha=3}^5\lambda_{\alpha}^3-2\lambda_{2}^3+\frac{3}{4}.
		\end{aligned}
		\end{equation}
		Here we used the facts
		$$\sum_{\alpha=2}^5\lambda_\alpha=\frac{3}{2}\,\, \text{and}\,\, \sum_{\alpha=2}^5\lambda_\alpha^2=\frac{3}{4}$$ in the last equality.
		
		\smallskip
		It follows from the appendix of Cheng-Zhou's paper \cite{Cheng-Zhou} that
		\ban
		-2\sigma_3={\sigma_1}^3-3\sigma_1\sigma_2-6\lambda_3\lambda_4\lambda_5,
		\ean
		where
		$\sigma_i=:\sum_{\alpha=1}^5\lambda_{\alpha}^i$ for $i=1,2,3$.
		Therefore, by using the facts
		$$\sum_{\alpha=3}^5\lambda_\alpha=\frac{3}{2}-\lambda_2\,\,\text{ and} \,\, \sum_{\alpha=3}^5\lambda_\alpha^2=\frac{3}{4}-\lambda_2^2$$ again, we have
		\ban
		-2\sum_{\alpha=3}^5\lambda_{\alpha}^3
		&=&\left( \sum_{\alpha=3}^5\lambda_{\alpha}\right)^3-3\left( \sum_{\alpha=3}^5\lambda_{\alpha}\right) \left( \sum_{\alpha=3}^5\lambda_{\alpha}\right)^2-6\lambda_3\lambda_4\lambda_5\\
		&=&(\frac{3}{2}-\lambda_2)^3-3(\frac{3}{2}-\lambda_2)(\frac{3}{4}-\lambda_2^2)-6\lambda_3\lambda_4\lambda_5.\\
		\ean
		It can be calculated directly that
		\ban
		&&\lambda_3\lambda_4\lambda_5\\
		 &=&(\lambda_{3}-\frac{1}{2})(\lambda_{4}-\frac{1}{2})(\lambda_{5}-\frac{1}{2})+\frac{1}{2}(\lambda_3\lambda_4+\lambda_3\lambda_5+\lambda_4\lambda_5)\\
		&&-\frac{1}{4} (\lambda_3+\lambda_4+\lambda_5)+\frac{1}{8}\\
		&=&(\lambda_{3}-\frac{1}{2})(\lambda_{4}-\frac{1}{2})(\lambda_{5}-\frac{1}{2})+
		\frac{1}{4}\left[ (\lambda_3+\lambda_4+\lambda_5)^2- (\lambda_3^2+\lambda_4^2+\lambda_5^2)\right] \\
		&&-\frac{1}{4} (\lambda_3+\lambda_4+\lambda_5)+\frac{1}{8}\\
		&=&(\lambda_{3}-\frac{1}{2})(\lambda_{4}-\frac{1}{2})(\lambda_{5}-\frac{1}{2})
		+\frac{1}{4}(\frac{3}{2}-\lambda_2)^2-(\frac{3}{4}-\lambda_2^2)\\
		&&-\frac{1}{4} (\frac{3}{2}-\lambda_2)+\frac{1}{8}\\
		&=&(\lambda_{3}-\frac{1}{2})(\lambda_{4}-\frac{1}{2})(\lambda_{5}-\frac{1}{2})
		+\frac{1}{2}\lambda_2(\lambda_2-1)+\frac{1}{8}.
		\ean
		Hence, plugging the above two equalities into \eqref{i}, we obtain
		\ban
		I&=&(\frac{3}{2}-\lambda_2)^3-3(\frac{3}{2}-\lambda_2)(\frac{3}{4}-\lambda_2^2)
		-6(\lambda_{3}-\frac{1}{2})(\lambda_{4}-\frac{1}{2})(\lambda_{5}-\frac{1}{2})\\
		&&-3\lambda_2(\lambda_2-1)-\frac{3}{4}+\frac{3}{4}\\
		&=&-\frac{3}{2}\lambda_2+6\lambda_2^2-6\lambda_2^3-6(\lambda_{3}-\frac{1}{2})(\lambda_{4}-\frac{1}{2})(\lambda_{5}-\frac{1}{2}).
		\ean
		
		Moreover, $|\lambda_{\alpha}-\frac{1}{2}|^2 \leq\lambda_2(1-\lambda_2)\leq\lambda_2$ holds for $\alpha=3,4,5$ due to \eqref{la}, and thus
		\ban
		I\leq-\frac{3}{2}\lambda_2+6\lambda_2^2-6\lambda_2^3+6\lambda_2^{\frac{3}{2}}.
		\ean
		We have completed the proof of \textbf{Claim 1}.
		
		\smallskip
		Second, we would like to handle the term
		\ban
		II:=-{\sum_{\alpha,\beta=2,\alpha\neq\beta}^5}( K_{1\alpha}+K_{1\beta})\lambda_\alpha\lambda_\beta.
		\ean
		\textbf{Claim 2}.
		\ban
		II\leq 2\frac{|\nabla Ric|^2}{f}+C(\lambda_1+\lambda_2)+\frac{1}{2}K_{12}.
		\ean
		In fact, we rewrite this term as follows
		\ban
		II&&=-{\sum_{\alpha,\beta=2,\alpha\neq\beta}^5}( K_{1\alpha}+K_{1\beta})\lambda_\alpha\lambda_\beta\\
		&&=-2{\sum_{\alpha,\beta=2,\alpha\neq\beta}^5} K_{1\alpha}\lambda_\alpha\lambda_\beta\\
		&&=-2{\sum_{\alpha=2}^5} K_{1\alpha}\lambda_\alpha(R-\lambda_\alpha)\\
		&&=-3{\sum_{\alpha=2}^5} K_{1\alpha}\lambda_\alpha+
		2{\sum_{\alpha=2}^5} K_{1\alpha}\lambda_\alpha^2\\
		&&=-3{\sum_{\alpha=2}^5} K_{1\alpha}\lambda_\alpha+2\left[ {\sum_{\alpha=2}^5} K_{1\alpha}(\lambda_\alpha-\frac{1}{2})^2+{\sum_{\alpha=2}^5} K_{1\alpha}\lambda_\alpha-\frac{1}{4}{\sum_{\alpha=2}^5} K_{1\alpha}\right]\\
		&&=-{\sum_{\alpha=2}^5} K_{1\alpha}\lambda_\alpha+2 {\sum_{\alpha=2}^5} K_{1\alpha}(\lambda_\alpha-\frac{1}{2})^2.
		\ean
		where $${\sum_{\alpha=2}^5} K_{1\alpha}=R_{11}=0$$ was used in the last equality.
		Due to Lemma \ref{lek1a} and \eqref{l345}, we obtain
		\ban
		-{\sum_{\alpha=2}^5} K_{1\alpha}\lambda_\alpha&=&
		{\frac{1}{f}}[(\lambda_2-\frac{1}{2})^2\lambda_2+{\sum_{\alpha=3}^5}(\lambda_\alpha-\frac{1}{2})^2\lambda_\alpha]\\
		&\leq& {\frac{1}{f}}\left[(\lambda_2-\frac{1}{2})^2\lambda_2+ \lambda_2(1-\lambda_2)(\frac{3}{2}-\lambda_2)\right]\\
		& \leq &C\frac{\lambda_1+\lambda_2}{f}.
		\ean
		It follows from \eqref{k1a} that
		\ban
		&&2{\sum_{\alpha=2}^5} K_{1\alpha}(\lambda_\alpha-\frac{1}{2})^2\\
		&=&2K_{12}(\lambda_2-\frac{1}{2})^2+2{\sum_{\alpha=3}^5} K_{1\alpha}(\lambda_\alpha-\frac{1}{2})^2\\
		&\leq& 2 K_{12}(\lambda_2^2-\lambda_2+\frac{1}{4})+{\sum_{\alpha=3}^5} K_{1\alpha}^2+{\sum_{\alpha=3}^5} (\lambda_\alpha-\frac{1}{2})^4\\
		&\leq& 2 K_{12}(\lambda_2-1)\lambda_2+\frac{1}{2} K_{12}+{\sum_{\alpha=3}^5} K_{1\alpha}^2+{\sum_{\alpha=3}^5} (\lambda_\alpha-\frac{1}{2})^4\\
		&\leq& K_{12}^2+(\lambda_2-1)^2\lambda_{2}^2+\frac{1}{2} K_{12}+{\sum_{\alpha=3}^5} K_{1\alpha}^2+\left[{\sum_{\alpha=3}^5}(\lambda_\alpha-\frac{1}{2})^2\right]^2\\
		&=& {\sum_{\alpha=2}^5}K_{1\alpha}^2+\left[(\lambda_2-1)^2\lambda_{2}^2+(1-\lambda_2)^2\lambda_{2}^2 \right]+\frac{1}{2}K_{12}\\
		&\leq&2\frac{|\nabla Ric|^2}{f}+C(\lambda_1+\lambda_2)+\frac{1}{2}K_{12}
		\ean
		for some constant $C$. Here the following inequality was used in the last equality,
		
		\begin{equation}\label{K1a^2}
		\begin{aligned}
		{\sum_{\alpha=2}^5}K_{1\alpha}^2=&\frac{1}{f^2}\sum_{\alpha=2}^5\left[ \nabla f\cdot \nabla \lambda_\alpha+\lambda_\alpha(\frac{1}{2}-\lambda_\alpha)\right]^2 \\
		\leq&\frac{2}{f^2}\left[ {\sum_{\alpha=2}^5}\left( \nabla f\cdot \nabla \lambda_\alpha\right)^2 + \sum_{\alpha=2}^5\lambda_\alpha^2(\frac{1}{2}-\lambda_\alpha)^2\right] \\
		\leq&2\frac{|\nabla Ric|^2}{f}+\frac{2}{f^2}{\sum_{\alpha=2}^5}\lambda_\alpha^2(\frac{1}{2}-\lambda_\alpha)^2\\
		\leq&2\frac{|\nabla Ric|^2}{f}+\frac{2}{f^2}\left[\lambda_2^2(\frac{1}{2}-\lambda_2)^2+\sum_{\alpha=3}^5\lambda_\alpha^2\lambda_2(1-\lambda_2) \right]\\
		\leq&2\frac{|\nabla Ric|^2}{f}+\frac{2}{f^2}\lambda_2\left[\lambda_2(\frac{1}{2}-\lambda_2)^2+(\frac{3}{4}-\lambda_2)(1-\lambda_2) \right]\\
		\leq&2\frac{|\nabla Ric|^2}{f}+\frac{C}{f^2}(\lambda_1+\lambda_2)
		\end{aligned}
		\end{equation}
		due to the boundedness of Ricci curvature.
		In conclusion, we obtain
		\ban
		II\leq 2\frac{|\nabla Ric|^2}{f}+C(\lambda_1+\lambda_2)+\frac{1}{2}K_{12}.
		\ean
		Thus, we have completed the proof of \textbf{Claim 2}.
		
		\smallskip
		\textbf{Claim 3}.
		\ban
		III&:=&\sum_{\alpha,\beta=2,\alpha\neq\beta}^5{\frac{1}{2f}}[({\frac{1}{2}}-\lambda_{\alpha })+({\frac{1}{2}}-\lambda_{\beta})]\lambda_\alpha\lambda_\beta\\
		&&-\frac{1}{f}\sum_{\alpha,\beta=2,\alpha\neq\beta}^5\left[ ({\frac{1}{2}}-\lambda_{\alpha})^2+({\frac{1}{2}}-\lambda_{\beta})^2\right] \lambda_\alpha\lambda_\beta\\
		 &&-\frac{2}{f}\sum_{\alpha,\beta=2,\alpha\neq\beta}^5({\frac{1}{2}}-\lambda_{\alpha})({\frac{1}{2}}-\lambda_{\beta})\lambda_\alpha\lambda_\beta\\
&&+\sum_{\alpha=2}^5\frac{2(\frac{1}{2}-\lambda_\alpha)^2\lambda_\alpha^2}{f}\\
		&\leq& \frac{C_2}{f}(\lambda_{1}+\lambda_{2}).
		\ean
		
		In fact, we rewrite the first two term of $III$ as follows,
		\ban
		&&\frac{1}{f}\sum_{\alpha,\beta=2,\alpha\neq\beta}^5({\frac{1}{2}}-\lambda_{\alpha})\lambda_\alpha\lambda_\beta
		-\frac{2}{f}\sum_{\alpha,\beta=2,\alpha\neq\beta}^5({\frac{1}{2}}-\lambda_{\alpha})^2\lambda_\alpha\lambda_\beta\\
		&=&\frac{1}{f}\sum_{\alpha=2}^5({\frac{1}{2}}-\lambda_{\alpha})\lambda_\alpha(\frac{3}{2}-\lambda_\alpha)
		-\frac{2}{f}\sum_{\alpha=2}^5({\frac{1}{2}}-\lambda_{\alpha})^2\lambda_\alpha(\frac{3}{2}-\lambda_\alpha)\\
		&=&\frac{1}{f}\sum_{\alpha=2}^5({\frac{1}{2}}-\lambda_{\alpha})\lambda_\alpha(\frac{3}{2}-\lambda_\alpha)(1-1+2\lambda_{\alpha})\\
		&=&\frac{2}{f}\left[\sum_{\alpha=2}^5\lambda_\alpha^2({\frac{1}{2}}-\lambda_{\alpha})+
		\sum_{\alpha=2}^5\lambda_\alpha^2({\frac{1}{2}}-\lambda_{\alpha})^2\right]\\
		&=&\frac{2}{f}\left[-\sum_{\alpha=2}^5\lambda_\alpha({\frac{1}{2}}-\lambda_{\alpha})^2+
		\frac{1}{2}\sum_{\alpha=2}^5\lambda_\alpha({\frac{1}{2}}-\lambda_{\alpha})
		+\sum_{\alpha=2}^5\lambda_\alpha^2({\frac{1}{2}}-\lambda_{\alpha})^2\right].\\
		\ean

 It follows from \eqref{l345} and \eqref{l345'} that
		\ban
		|\sum_{\alpha=2}^5\lambda_\alpha({\frac{1}{2}}-\lambda_{\alpha})^2|
		&\leq&\lambda_2({\frac{1}{2}}-\lambda_2)^2+\lambda_2(1-\lambda_2)(\frac{3}{2}-\lambda_2)\\
		&=&\lambda_2\left[ ({\frac{1}{2}}-\lambda_2)^2+(1-\lambda_2)(\frac{3}{2}-\lambda_2)\right]
		\ean
		and
		\ban
		\sum_{\alpha=2}^5\lambda_\alpha^2({\frac{1}{2}}-\lambda_{\alpha})^2
		&\leq&\lambda_2^2({\frac{1}{2}}-\lambda_2)^2+\lambda_2(1-\lambda_2)(\frac{3}{4}-\lambda_2^2)\\
		&=&\lambda_2\left[\lambda_2({\frac{1}{2}}-\lambda_2)^2+(1-\lambda_2)(\frac{3}{4}-\lambda_2^2)\right].
		\ean

Since $\sum_{\alpha=2}^5\lambda_\alpha=R=2|Ric|^2=2\sum_{\alpha=2}^5\lambda_\alpha^2$, so $\sum_{\alpha=2}^5\lambda_\alpha({\frac{1}{2}}-\lambda_{\alpha})=0$.

	\noindent{	Hence},
		\ban
		&&\frac{1}{f}\sum_{\alpha,\beta=2,\alpha\neq\beta}^5({\frac{1}{2}}-\lambda_{\alpha})\lambda_\alpha\lambda_\beta
		-\frac{2}{f}\sum_{\alpha,\beta=2,\alpha\neq\beta}^5({\frac{1}{2}}-\lambda_{\alpha})^2\lambda_\alpha\lambda_\beta\\
		&\leq& \frac{C}{f}(\lambda_{1}+\lambda_{2})
		\ean
		for some constant $C$, since the Ricci curvature is nonnegative and bounded.
		
		We consider the third term of $III$ as follows,
		\ban
		&&-\sum_{\alpha,\beta=2,\alpha\neq\beta}^5({\frac{1}{2}}-\lambda_{\alpha})({\frac{1}{2}}-\lambda_{\beta})\lambda_\alpha\lambda_\beta\\
		&=&-\frac{1}{2}\sum_{\alpha,\beta=2,,\alpha\neq\beta}^5({\frac{1}{2}}-\lambda_{\alpha})\lambda_\alpha\lambda_\beta
		+\sum_{\alpha,\beta=2,\alpha\neq\beta}^5({\frac{1}{2}}-\lambda_{\alpha})\lambda_\alpha\lambda_\beta^2\\
		&=&-\frac{1}{2}\sum_{\alpha=2}^5({\frac{1}{2}}-\lambda_{\alpha})\lambda_\alpha(\frac{3}{2}-\lambda_\alpha)
		+\sum_{\alpha=2}^5({\frac{1}{2}}-\lambda_{\alpha})\lambda_\alpha(\frac{3}{4}-\lambda_\alpha^2)\\
		&=&-\frac{3}{4}\sum_{\alpha=2}^5({\frac{1}{2}}-\lambda_{\alpha})\lambda_\alpha
		+\frac{1}{2}\sum_{\alpha=2}^5({\frac{1}{2}}-\lambda_{\alpha})\lambda_\alpha^2\\
		&&+\frac{3}{4}\sum_{\alpha=2}^5({\frac{1}{2}}-\lambda_{\alpha})\lambda_\alpha
		-\sum_{\alpha=2}^5({\frac{1}{2}}-\lambda_{\alpha})\lambda_\alpha^3\\
		&=&\sum_{\alpha=2}^5\lambda_\alpha^2({\frac{1}{2}}-\lambda_{\alpha})^2\\
		&\leq&\lambda_2\left[\lambda_2({\frac{1}{2}}-\lambda_2)^2+(1-\lambda_2)(\frac{3}{4}-\lambda_2^2)\right],
		\ean
		which implies that
		\[
		 -\frac{2}{f}\sum_{\alpha,\beta=2,,\alpha\neq\beta}^5({\frac{1}{2}}-\lambda_{\alpha})({\frac{1}{2}}-\lambda_{\beta})\lambda_\alpha\lambda_\beta
		\leq\frac{C}{f}(\lambda_{1}+\lambda_{2}).
		\]
		for some constant $C$, since the Ricci curvature is nonnegative and bounded.

    Therefore we have
		\[
		III\leq\frac{C}{f}(\lambda_{1}+\lambda_{2}).
		\]
		We have completed the proof of \textbf{Claim 3}.
		
		\smallskip
		Finally, we consider the term $IV:=2\sum_{\alpha,\beta=2}^5W^{\Sigma(s)}_{\alpha\beta}\lambda_\alpha\lambda_\beta$.
		
		\smallskip
		\textbf{Claim 4}.
		\ban
		IV\leq\frac{17}{4}\left[ \epsilon(\lambda_1+\lambda_2)+\frac{1}{\epsilon}|W^{\Sigma(s)}|^2\right]
		\ean
		for any constant $\epsilon>0$ satisfying $\frac{17}{4}\epsilon\leq\frac{1}{4}$.
		
		In fact, we rewrite the first two term of $IV$ as follows,
		\begin{equation*}
		\begin{aligned}
		IV
		=&4W^{\Sigma(s)}_{23}\lambda_{2}\lambda_{3}+4W^{\Sigma(s)}_{24}\lambda_{2}\lambda_{4}+4W^{\Sigma(s)}_{25}\lambda_{2}\lambda_{5}\\
		&+4W^{\Sigma(s)}_{34}\lambda_{3}\lambda_{4}+4W^{\Sigma(s)}_{35}\lambda_{3}\lambda_{5}+4W^{\Sigma(s)}_{45}\lambda_{4}\lambda_{5}\\
		=&4\lambda_{2}(W^{\Sigma(s)}_{23}\lambda_{3}+W^{\Sigma(s)}_{24}\lambda_{4}+W^{\Sigma(s)}_{25}\lambda_{5})\\
		&+2\lambda_3(W^{\Sigma(s)}_{34}\lambda_{4}+W^{\Sigma(s)}_{35}\lambda_{5})\\
		&+2\lambda_4(W^{\Sigma(s)}_{34}\lambda_{3}+W^{\Sigma(s)}_{45}\lambda_{5})\\
		&+2\lambda_5(W^{\Sigma(s)}_{35}\lambda_{3}+W^{\Sigma(s)}_{45}\lambda_{4}).
		\end{aligned}
		\end{equation*}
		Using the fact that
		$\sum_{\alpha=2}^5W^{\Sigma(s)}_{3\alpha}=0$ again, we have
		\ban
		&&2\lambda_3(W^{\Sigma(s)}_{34}\lambda_{4}+W^{\Sigma(s)}_{35}\lambda_{5})\\
		 &=&2\lambda_3\left[W^{\Sigma(s)}_{34}(\lambda_{4}-\frac{1}{2})+W^{\Sigma(s)}_{35}(\lambda_{5}-\frac{1}{2})-\frac{1}{2}W^{\Sigma(s)}_{32}\right] \\
		 &\leq&2(|W^{\Sigma(s)}_{34}|^2+|W^{\Sigma(s)}_{35}|^2)^\frac{1}{2}\left[(\lambda_{4}-\frac{1}{2})^2+(\lambda_{5}-\frac{1}{2})^2\right]^\frac{1}{2}-W^{\Sigma(s)}_{23}\lambda_{3}\\
		&\leq&2|W^{\Sigma(s)}|\lambda_2^{\frac{1}{2}}-W^{\Sigma(s)}_{23}\lambda_{3}\\
		&\leq&\left( \epsilon\lambda_2+\frac{1}{\epsilon}|W^{\Sigma(s)}|^2\right)-W^{\Sigma(s)}_{23}\lambda_{3}
		\ean
		for any $\epsilon>0$, where we used the fact $$(\lambda_{4}-\frac{1}{2})^2+(\lambda_{5}-\frac{1}{2})^2 \leq\lambda_2(1-\lambda_2)\leq\lambda_2$$  in the second inequality.
		Similarly,
		\ban
		2\lambda_4(W^{\Sigma(s)}_{34}\lambda_{3}+W^{\Sigma(s)}_{45}\lambda_{5})&\leq&\left( \epsilon\lambda_2+\frac{1}{\epsilon}|W^{\Sigma(s)}|^2\right)-W^{\Sigma(s)}_{24}\lambda_{4}
		\ean
		and
		\ban
		2\lambda_5(W^{\Sigma(s)}_{35}\lambda_{3}+W^{\Sigma(s)}_{45}\lambda_{4})&\leq&\left( \epsilon\lambda_2+\frac{1}{\epsilon}|W^{\Sigma(s)}|^2\right)-W^{\Sigma(s)}_{25}\lambda_{5}.
		\ean
		Therefore,
		\ban
		IV&\leq&4\lambda_{2}\sum_{\alpha=3}^5 W^{\Sigma(s)}_{2\alpha}\lambda_{\alpha}+3\left( \epsilon\lambda_2+\frac{1}{\epsilon}|W^{\Sigma(s)}|^2\right)-\sum_{\alpha=3}^5 W^{\Sigma(s)}_{2\alpha}\lambda_{\alpha}.
		\ean

		Next, we would like to deal with the term $\sum_{\alpha=3}^5 W^{\Sigma(s)}_{2\alpha}\lambda_{\alpha}$. Since the Weyl curvature is tracefree,  we see
		\ban
		\sum_{\alpha=3}^5W^{\Sigma(s)}_{2\alpha}=\sum_{\alpha=3}^5W^{\Sigma(s)}_{2\alpha2\alpha}=0,
		\ean
		and then
		\begin{equation}\label{w2a}
		\begin{aligned}
		&|W^{\Sigma(s)}_{23}\lambda_{3}+W^{\Sigma(s)}_{24}\lambda_{4}+W^{\Sigma(s)}_{25}\lambda_{5}|\\
		=&|W^{\Sigma(s)}_{23}(\lambda_{3}-\frac{1}{2})+W^{\Sigma(s)}_{24}(\lambda_{4}-\frac{1}{2})+W^{\Sigma(s)}_{25}(\lambda_{5}-\frac{1}{2})|\\
		\leq&|W^{\Sigma(s)}|[(\lambda_{3}-\frac{1}{2})^2+(\lambda_{4}-\frac{1}{2})^2+(\lambda_{5}-\frac{1}{2})^2]^{\frac{1}{2}}\\
		=&|W^{\Sigma(s)}|[\lambda_{2}(1-\lambda_2)]^{\frac{1}{2}}\\
		\leq&|W^{\Sigma(s)}|\lambda_{2}^{\frac{1}{2}}\\
		\leq&\frac{1}{2}\left( \epsilon\lambda_2+\frac{1}{\epsilon}|W^{\Sigma(s)}|^2\right).
		\end{aligned}
		\end{equation}
		Therefore, we see
		\ban
		IV&\leq&2\lambda_{2}\left( \epsilon\lambda_2+\frac{1}{\epsilon}|W^{\Sigma(s)}|^2\right)+3\left( \epsilon\lambda_2+\frac{1}{\epsilon}|W^{\Sigma(s)}|^2\right)\\
		&&+\frac{1}{2}\left( \epsilon\lambda_2+\frac{1}{\epsilon}|W^{\Sigma(s)}|^2\right)\\
		&\leq&\frac{17}{4}\left( \epsilon\lambda_2+\frac{1}{\epsilon}|W^{\Sigma(s)}|^2\right),
		\ean
		due to the fact $4\lambda_{2}\leq R=\frac{3}{2}$.  We have completed the proof of \textbf{Claim 4}.
		
		\vspace{0.5cm}
		Consequently, from \textbf{Claims 1--4}, we obtain that
		\ban
		|\nabla Ric|^2&=&I+II+III+IV\\
		&\leq&-\frac{3}{2}(\lambda_1+\lambda_2)+6(\lambda_1+\lambda_2)^2-6(\lambda_1+\lambda_2)^3+6(\lambda_1+\lambda_2)^{\frac{3}{2}}\\
		&&+2\frac{|\nabla Ric|^2}{f}+C(\lambda_1+\lambda_2)+\frac{1}{2}K_{12}+\frac{C_2}{f}(\lambda_{1}+\lambda_{2})\\
		&&+\frac{17}{4}\left( \epsilon(\lambda_{1}+\lambda_{2})+\frac{1}{\epsilon}|W^{\Sigma(s)}|^2\right)\\
		&\leq&2\frac{|\nabla Ric|^2}{f}+C(\lambda_1+\lambda_2)+\frac{1}{2}K_{12}+C|W^{\Sigma(s)}|^2,
		\ean
		for some constant $C$ and small $\epsilon$ satisfying $\frac{17}{4}\epsilon\leq\frac{1}{4}$.
		
		\smallskip
		Hence
		\ban
		|\nabla Ric|^2\leq C(\lambda_1+\lambda_2)+K_{12}+C|W^{\Sigma(s)}|^2
		\ean
		on $M\setminus D(a)$, where $C$ and $a>0$ are constants. We have completed the proof of Proposition \ref{key}.
		
	\end{proof}
	
	
	\section{Curvature bound and uniform decay of $\lambda_1+\lambda_2$}\label{sec5}
	
	In this section, based the point-picking argument, we will prove the Riemannian curvature is bounded.
	This is because $\int_{\Sigma}|W|^2d\sigma$ tends to zero at infinity, which implies the blowing up limit must be flat. By the similar argument, we can also prove that $\lambda_1+\lambda_2 \rightarrow 0$ and $\nabla_{\nabla f}Ric$ also tend to zero.
	
	\smallskip
	
	In order to prove the curvature bound, we continue to handle the term $K_{12}$ in Proposition \ref{key} and have the following lemma.
	\begin{lem}\label{ricci'}
		Let $(M^5, g, f)$ be a five-dimensional shrinking gradient Ricci soliton with constant scalar curvature $R=\frac{3}{2}$. There is some constant $C>0$, such that
		\[
		|\nabla Ric|^2\leq  C(\lambda_1+\lambda_2)+C|W^{\Sigma(s)}|^2+C
		\]
		on $M\setminus D(a)$ for some $a>0$ and $s\geq a$.	
	\end{lem}
	\begin{proof}
		We need to rewrite the equation in Propositon \ref{key} in the following way.
		\ban
		|\nabla Ric|^2&\leq& C(\lambda_1+\lambda_2)+K_{12}+C|W^{\Sigma(s)}|^2\\
		&\leq& C(\lambda_1+\lambda_2)+K_{12}^2+\frac{1}{16}+C|W^{\Sigma(s)}|^2\\
		\ean
		From equation \eqref{K1a^2}, we have
		\[
		K_{12}^2\leq2\frac{|\nabla Ric|^2}{f}+\frac{C}{f^2}(\lambda_1+\lambda_2)
		\]
		Thus, we have
		\[
		|\nabla Ric|^2\leq C(\lambda_1+\lambda_2)+2\frac{|\nabla Ric|^2}{f}+\frac{1}{16}+C|W^{\Sigma(s)}|^2,
		\]
		which yields
		\ban
		|\nabla Ric|^2\leq  C(\lambda_1+\lambda_2)+C|W^{\Sigma(s)}|^2+C
		\ean
		for some constant $C$.
	\end{proof}
	
	Next, we have the relationship between $\int_{\Sigma(s)}{|W^{\Sigma(s)}|^2}d\sigma$ and the Ricci curvature by the Gauss-Bonnet-Chern formula.	
\begin{lem}\label{gbc}
	Let $(M^5, g, f)$ be a five-dimensional shrinking gradient Ricci soliton with constant scalar curvature $R=\frac{3}{2}$. Then we have
	\begin{equation}
		\int_{\Sigma(s)}{|W^{\Sigma(s)}|^2}d\sigma
		=2\int_{\Sigma(s)}{\frac{|\nabla_{\nabla f} Ric|^2}{f^2}}d\sigma\leq2\int_{\Sigma(s)}{\frac{|\nabla Ric|^2}{f}}d\sigma.
	\end{equation}	
\end{lem}

\begin{proof}
	Recall the Gauss-Bonnet-Chern formula:
	for a closed, oriented four-dimensional Riemannian manifold $(M^4,\,g)$,
	\ban
	\chi(M^4)=\frac{1}{4\pi^2}\int_{M^4}\left[ \frac{1}{8}|W_g|^2-\frac{1}{12}\left(3|\text{Ric}_g|^2-R_g^2\right) \right]dV,
	\ean	
	where $\chi(M^4)$ is the Euler characteristic of $M^4$, $W_g$, $R_g$, $\text{Ric}_g$ are its Weyl curvature, scalar curvature and Ricci curvature, respectively and $dV$ is the volume element of $M^4$.
	
	On a five-dimensional simply connected shrinking Ricci soliton with $R=\frac{3}{2}$, it follows from Proposition \ref{levelset} that the zero set $f^{-1}(0)$ is a three-dimensional simply connected closed manifold. Actually, $f^{-1}(0)$ is the deformation contraction of $M^5$,  hence has to be diffeomorphic to $\mathbb{S}^3$. Because $f(x)=\frac{1}{4}d(x, f^{-1}(0))^2$, $f=|\nabla f|^2$ and the normal  exponential map is a local diffeomorphism, it is easy to see that $f^{-1}(s)$ is diffeomorphic to $\mathbb{S}^1\times \mathbb{S}^3$ when $s$ is small and positive. Hence all the level set of $f$ are diffeomorphic to $\mathbb{S}^1\times \mathbb{S}^3$ since $f$ has no critical point away from $f^{-1}(0)$. Therefore, $\chi(\Sigma(s))=\chi(\mathbb{S}^3\times \mathbb{S}^1)=0$, where $\Sigma(s)=f^{-1}(s)$ for $s>0$.
	
	As for $\int_{\Sigma(s)}{|W^{\Sigma(s)}|^2}d\sigma$, putting equations \eqref{rsigma} and \eqref{ricsigma'} into the Gauss-Bonnet formula, we get
	
	\ban
	&&\int_{\Sigma(s)}{|W^{\Sigma(s)}|^2}d\sigma\\
	&=&\int_{\Sigma(s)}2 \left[ |Ric^{\Sigma(s)}|^2-\frac{1}{3}(R^{\Sigma(s)})^2\right]d\sigma +32\pi^2\chi(\Sigma(s))\\
	&=&\int_{\Sigma(s)}2\left( |Ric|^2+{\frac{|\nabla_{\nabla f} Ric|^2}{f^2}}-\frac{3}{4}\right) d\sigma\\
	&=&2\int_{\Sigma(s)}{\frac{|\nabla_{\nabla f} Ric|^2}{f^2}}d\sigma\\
	&\leq&2\int_{\Sigma(s)}{\frac{|\nabla Ric|^2|\nabla f|^2}{f^2}}d\sigma\\
	&=&2\int_{\Sigma(s)}{\frac{|\nabla Ric|^2}{f}}d\sigma,	
	\ean
	where the fact $|\nabla f|^2=f$ was used in the last equality.	
\end{proof}

Combining the above two lemmas, we have the following proposition.
\begin{prop}Suppose $(M^5, g, f)$ is a five-dimensional shrinking gradient Ricci soliton with $R=\frac{3}{2}$, then
	\ban
	\int_{\Sigma(s)}|W^{\Sigma(s)}|^2d\sigma\leq \int_{\Sigma(s)}\frac{C}{s}d\sigma\leq \frac{C}{\sqrt{s}}.
	\ean
\end{prop}

\begin{proof}
	It follows from the above two lemmas that
	\ban
	\int_{\Sigma(s)}|W^{\Sigma(s)}|^2d\sigma&\leq& 2\int_{\Sigma(s)}{\frac{|\nabla Ric|^2}{f}}d\sigma\\
	&\leq&2\int_{\Sigma(s)}\frac{1}{s}\left[ C(\lambda_{1}+\lambda_{2})+C|W^{\Sigma(s)}|^2+C\right] d\sigma\\
	&\leq&2\int_{\Sigma(s)}\frac{C}{s}d\sigma +\int_{\Sigma(s)} \frac{C}{s}|W^{\Sigma(s)}|^2d\sigma.
	\ean
 Hence
	\ban
	\int_{\Sigma(s)}|W^{\Sigma(s)}|^2d\sigma\leq \int_{\Sigma(s)}\frac{C}{s}d\sigma
	\ean
	for sufficiently large $s$.
	
	From item (v) in Poposition \ref{levelset}, the volume of $\Sigma(s)$ satisfies
	\ban
	\text{Vol}(\Sigma(s))=c\sqrt{s}
	\ean
	for some constant $c$.
	Hence
	\ban
	\int_{\Sigma(s)}\frac{C}{s}d\sigma\leq \frac{C}{\sqrt{s}},
	\ean
	and we have completed the proof of this proposition.
\end{proof}

In order to prove the curvature bound, we recall the following results.
\begin{lem}[\cite{Chen-Zhu}]\label{point picking}
	Given a complete noncompact Riemannian manifold with unbounded curvature, we can find a sequence of points $\{p_j\}_{j=1}^\infty$ divergent to infinity such that for each positive integer $j$, we have $|Rm(p_j)|\geq j$ and
	\ban
	|Rm(x)|\leq 4 |Rm(p_j)|
	\ean
	for $x\in B(p_j, \frac{j}{\sqrt{|Rm(p_j)|}})$.
\end{lem}

The following backward pseudolocality Theorem will also be useful.
\begin{theo}[\cite{Li-Wang2}]\label{backward pseudolocality} For any $\alpha>0$, there is $\epsilon(n, \alpha)$ such that the following holds.

 Let $(M^n, g(t))_{t\in I}$ be a Ricci flow induced by a shrinking gradient Ricci soliton. Given $(x_0, t_0)\in M\times I$ and $r>0$, if
 \ban
 |B_{t_0}(x_0, r)|\geq \alpha r^n, \quad |Rm|\leq (\alpha r)^{-2} \quad on \quad B_{t_0}(x_0, r),
 \ean
 then
 \ban
 |Rm|\leq (\epsilon r)^{-2}  \quad on \quad P(x_0, t_0; (1-\alpha)r, -(\epsilon r)^2, 0).
 \ean
\end{theo}

\begin{theo}\label{bounded curvature}Suppose $(M^5, g, f)$ is a five-dimensional shrinking gradient Ricci soliton with $R=\frac{3}{2}$, then its curvature is bounded.
\end{theo}

\begin{proof}
	Suppose not, by Lemma \ref{point picking}, then there exists a sequence of points $\{p_j\}_{j=1}^\infty$ divergent to infinity such that for each positive integer $j$, we have $|Rm(p_j)|\geq j$ and
	\ban
	|Rm(x)|\leq 4 |Rm(p_j)|
	\ean
	for $x\in B(p_j, \frac{j}{\sqrt{|Rm(p_j)|}})$.
	
	By the $\kappa$-noncollapsed theorem in \cite{Li-Wang} and the scalar curvature is bounded for $(M, g, f)$, we get that $\text{Vol}(B(p_j, r))\geq \kappa r^5$ for $0<r<1$ and some  $\kappa$ depending only on the soliton.

For the shrinking gradient Ricci soliton  $(M,g,f)$,
 the associated Ricci flow $(M, g(t), p_j)$ is defined on $(-\infty, 0)$, where
\[
g(t):=(-t)\phi^*_tg,\quad \frac{d\phi_t}{dt}=\frac{\nabla f}{-t}
\quad  \text{and}\quad \phi_{-1}=\mathrm{Id}.
\]
Now we can use Theorem \ref{backward pseudolocality} with $r=|Rm|(p_j)^{-\frac{1}{2}}$ to derive that
\ban
|Rm|\leq (\epsilon r)^{-2} \quad on \quad B(p_j, g(-1), \frac{j}{2}r) \times [-1-(\epsilon r)^2, -1].
\ean
Then we can apply Hamilton's compactness theorem to obtain that the rescaled manifolds $\left(B(p_j, g, \frac{j}{2\sqrt{|Rm(p_j)|}}),  |Rm(p_j)|g, p_j\right)$ converge to a smooth complete Riemannian manifold $(M_\infty, g_\infty, p_\infty)$ with $|Rm(p_\infty)|=1$ which is Ricci flat because $(M^5, g)$ has bounded Ricci curvature and $|Rm(p_j)|\rightarrow \infty$, moreover  $(M_\infty, g_\infty, p_\infty)$ has Euclidean volume growth.
	
	\smallskip
	Since the integral curves of $f$ passing through $p_j$ is a geodesic with respect to $(M, g)$, the geodesic segment of these curves contained in $B\left( p_j, g,  \frac{j}{2\sqrt{|Rm(p_j)|}}\right) $ will converge to a geodesic line in $(M_\infty, g_\infty)$, then Cheeger-Gromoll's splitting theorem \cite{Cheeger-Gromoll} implies that $M_\infty =\mathbb{R}\times N^4$, where $N^4$ is Ricci flat and of Euclidean Volume growth.
	
	\smallskip
	Because $(\Sigma(s), g)$ has second fundamental form
	\ban
	h=\frac{\frac{1}{2}g-Ric}{|\nabla f|},
	\ean
	which tends to zero as $f\rightarrow \infty$,
	so the second fundamental form of $B\left( p_j, g, \frac{j}{2\sqrt{|Rm(p_j)|}}\right)   \cap \Sigma(f(p_j))$ with metric $|Rm(p_j)|g$ converges to zero. This implies the level set $B\left( p_j, g, \frac{j}{2\sqrt{|Rm(p_j)|}}\right)   \cap \Sigma(f(p_j))$ with the induced rescaled metrics $|Rm(p_j)|g$ will converge to $N^4$.
	
	\smallskip
	On the  other hand, we have the key estimate
	\ban
	\int_{\Sigma(s)}|W^{\Sigma(s)}|^2d\sigma\leq \frac{C}{\sqrt{s}}\rightarrow 0
	\ean
	as $s\rightarrow\infty$.
	
	\smallskip
	It is known that integral of the Weyl curvature is scaling invariant in dimension $4$.
	All the above implies that
	\ban
	\int_{B(p_j, |Rm(p_j)|g, \frac{j}{2})  \cap \Sigma(f(p_j))}|W^{\Sigma(f(p_j))}|^2d\sigma\leq  \frac{C}{\sqrt{f(p_j)}}\rightarrow 0.
	\ean
	So $N^4$ has vanishing Weyl curvature, hence it is flat. This contradicts the fact that $|Rm(p_\infty)|=1$.
\end{proof}

\begin{theo}
	Suppose $(M^5, g, f)$ is a five-dimensional shrinking gradient Ricci soliton with $R=\frac{3}{2}$, then $\lambda_1+\lambda_2\rightarrow 0$ at infinity.
\end{theo}

\begin{proof}
	Suppose on the contrary, then there exists a sequence of $q_j$ divergent to infinity with $(\lambda_1+\lambda_2)(q_j)\geq \delta$ for some $\delta>0$.
	
	\smallskip
	As \cite{Naber}, define $f_j(x)=\frac{f(x)-f(q_j)}{|\nabla f(q_j)|}$. By Theorem \ref{bounded curvature}, we have the curvature is bounded, so the pointed manifolds $(M, g, q_j)$ converge in Cheeger-Gromov sense to $(M_\infty, g_\infty, q_\infty)$. Since  $|\nabla f_j(q_j)|=1$,
	\ban
	\nabla^2 f_j=\frac{\nabla^2 f}{|\nabla f(q_j)|}=\frac{\frac{1}{2}g-Ric}{|\nabla f(q_j)|}
	\ean
	tends to zero at infinity,  $f_i$ converges to a smooth function $f_\infty$ with $|\nabla f_\infty|(q_\infty)=1$ and $\nabla^2 f_\infty =0$. Hence $M_\infty=\mathbb{R}\times N^4$, where $N^4$ is a four-dimensional complete Riemannian manifold with $Ric\geq 0$ and of constant scalar curvature $\frac{3}{2}$.
	
	\smallskip
	Again as the above theorem,  $(\Sigma(s), g)$ has second fundamental form
	\ban
	h=\frac{\frac{1}{2}g-Ric}{|\nabla f|},
	\ean
	which tends to zero as $f\rightarrow \infty$;
	this implies the level set $\Sigma(f(q_j))$  with the induced induced metric will converge to $N^4$ with $\lambda_1(g_{N^4})\geq \delta$.
	
	\smallskip
	The key estimate gives
	\ban
	\int_{\Sigma(f(q_j))}|W^{\Sigma(f(q_j))}|^2d\sigma\leq \frac{C}{\sqrt{f(q_j)}}\rightarrow 0
	\ean
	as $s\rightarrow\infty$.
	So $N^4$ has vanishing weyl curvature.
	
	\smallskip
	Thanks to the classification of complete locally conformally flat manifolds
	with nonnegative Ricci curvature by Zhu \cite{Zhu} and Carron-Herzlich \cite{Carron-Herzlic}, $N^4$ is
	one of the following:
	
	\smallskip
	(1) $N$ is non-flat and globally conformally equivalent to $\mathbb{R}^4$;
	
	\smallskip
	(2)  $N$ is globally conformally equivalent to a space form of positive curvature;
	
	\smallskip
	(3)  $N$ is locally isometric to the cylinder $\mathbb{R}\times \mathbb{S}^3$;
	
	\smallskip
	(4)   $N$ is isometric to a complete flat manifold.\\
	
	If  case (1) happens, then there is a positive function $u$ such that $g_N=u^2 g_{E}$ has constant scalar curvature $\frac{3}{2}$, where $g_E$ is the Euclidean metric.  Equivalently,
	\ba\label{zheng}
	\Delta u+\frac{1}{4} u^3 =0 \quad on \quad  \mathbb{R}^4.
	\ea
	By  Caffarelli-Gidas-Spruck \cite{Caffarelli-Gidas-Spruck} or Chen-Li \cite{Chen-Li}, the solution to (\ref{zheng}) has been classified, and none is complete.  Contradiction. The case (2) is impossible, since $N$ is noncompact. If case (3) happens,  it contradicts with $\lambda_1(g_{N^4})\geq \delta$.  Case (4) can not happen, since $N$ is nonflat.
	
	In all, the Weyl curvature of $N^4$ couldn't be zero, contradiction.
\end{proof}

Similarly, we have the following corollary.
\begin{coro}
	Suppose $(M^5, g, f)$ is a five-dimensional shrinking gradient Ricci soliton with $R=\frac{3}{2}$, then
	\[
	\nabla_{\nabla_f}Ric \rightarrow 0	
	\]
	and
	\ban
	Ric-2 Rm*Ric\rightarrow 0
	\ean
	at infinity.
\end{coro}
\begin{proof}
	The proof is similar to the above theorem, and notice that
	\ban
	\Delta Ric=0
	\ean
	and
	\ban
	Ric-Rm*Ric=0
	\ean
	on $\mathbb{R}\times \mathbb{R}\times \mathbb{S}^3$.
	
	Combining with
	\ban
	\Delta_f Ric=\Delta Ric -\nabla_{\nabla f}Ric=Ric-Rm*Ric
	\ean
	gives the desired result.
	
\end{proof}

\section{The Proof of Theorem \ref{main}}\label{sec6}

In this section, we improve the estimate on $\int_{\Sigma(s)}|W^{\Sigma(s)}|^2d\sigma$, then by the integral method we show $\lambda_1+\lambda_2=0$, thereby proving Theorem \ref{main}.

\begin{prop}\label{first ineqality}
	Let $(M^5, g, f)$ be a five-dimensional shrinking gradient Ricci soliton with $R=\frac{3}{2}$. Denote $u=\lambda_1+\lambda_2$, $h=f+\frac{3}{2}\log f-\frac{40}{f}$, then
	\ban
	\int_{M\setminus D(a)} \Delta_h u\cdot e^{-h}dvol
	\leq -0.8\int_{M\setminus D(a)}u\cdot e^{-h} dvol
	\leq  0
	\ean
for almost everywhere sufficiently large $a$.
\end{prop}

\begin{proof}
	First, we claim that
	\begin{equation}\label{intw}
		\begin{aligned}
			&\int_{M\setminus D(a)} |W^{\Sigma(s)}|^2\cdot e^{-h}dvol\\
			\leq& \int_{M\setminus D(a)}\left[\frac{C(\lambda_1+\lambda_2)}{f}+\frac{2\nabla f\cdot\nabla(\lambda_1+\lambda_2)}{f^2}\right]\cdot e^{-h}dvol.
		\end{aligned}
	\end{equation}
	In fact, due to $\lambda_1+\lambda_2\rightarrow 0$ at infinity, it follows from Proposition \ref{key} that
	\ban
	|\nabla Ric|^2\leq C(\lambda_1+\lambda_2)+K_{12}+C|W^{\Sigma(s)}|^2,
	\ean		
	for some constant $C$. Then by Lemma \ref{gbc}
	\ban
	\int_{\Sigma(s)}|W^{\Sigma(s)}|^2d\sigma_{\Sigma(s)}
	\leq \int_{\Sigma(s)}\left[ \frac{C(\lambda_1+\lambda_2)}{f}+\frac{K_{12}}{f}+\frac{C|W^{\Sigma(s)}|^2}{f}\right]d\sigma_{\Sigma(s)}.
	\ean
	By the absorbing inequality,
	\ban
	\int_{\Sigma(s)}(1-\frac{C}{f})|W^{\Sigma(s)}|^2
	\leq \int_{\Sigma(s)} \left[ \frac{C(\lambda_1+\lambda_2)}{f}+\frac{K_{12}}{f}\right] d\sigma_{\Sigma(s)}.
	\ean
	for $s\geq a$ with $a$ large. Thus, we immediately derive that
	\ban
	\int_{\Sigma(s)}|W^{\Sigma(s)}|^2d\sigma_{\Sigma(s)}
	\leq \int_{\Sigma(s)}\left[  \frac{C(\lambda_1+\lambda_2)}{f}+2\frac{K_{12}}{f}\right] d\sigma_{\Sigma(s)}.
	\ean
	Combining this with eqaution \eqref{k1a}, we see that
	\ban
	&&\int_{M\setminus D(a)} |W^{\Sigma(s)}|^2\cdot e^{-h}dvol\\
	&\leq& \int_{M\setminus D(a)}\left[\frac{C(\lambda_1+\lambda_2)}{f}+2\frac{K_{12}}{f}\right]\cdot e^{-h}dvol\\
	&\leq& \int_{M\setminus D(a)}\left[\frac{C(\lambda_1+\lambda_2)}{f}+\frac{2\nabla f\cdot\nabla(\lambda_1+\lambda_2)}{f^2}\right]\cdot e^{-h}dvol.
	\ean
	
	\vspace{0.5cm}
Since the curvature is bounded by Theorem \ref{bounded curvature}, so $|\nabla Ric|$ is also bounded by Shi's derivative estimate, this gives $|\frac{\nabla f\cdot \nabla (\lambda_1+\lambda_2)^2}{f}|\leq C\cdot\frac{|\nabla f|\cdot|\nabla Ric|}{f}\cdot (\lambda_1+\lambda_2)\leq \delta (\lambda_1+\lambda_2)$, where $\delta$ is sufficiently small. 	Noticing that $\lambda_1+\lambda_2\rightarrow 0$ at infinity and substituting inequality \eqref{w2a} into Proposition \ref{lap12}, we get
	\begin{equation*}
		\Delta_f(\lambda_1+\lambda_2)
		\leq-0.9(\lambda_1+\lambda_2)+\frac{3}{2f}\nabla f\cdot \nabla (\lambda_1+\lambda_2)+ \epsilon|W^{\Sigma(s)}|^2+\frac{1}{\epsilon}(\lambda_1+\lambda_2)
	\end{equation*}	
	for any $\epsilon>0$. Together with \eqref{intw}, taking $\epsilon=20$, it is easy to see that
	
\ban
	&&\Delta_f(\lambda_1+\lambda_2)\\
	&\leq&-0.9(\lambda_1+\lambda_2)+{\frac{3\nabla f\cdot \nabla(\lambda_1+\lambda_2)}{2f}}
	+20 |W^{\Sigma(s)}|^2
	+{\frac{1}{20}} (\lambda_1+\lambda_2)\\
	&\leq&-0.85(\lambda_1+\lambda_2)+\frac{3}{2}\nabla  \log f\cdot \nabla(\lambda_1+\lambda_2)+20 |W^{\Sigma(s)}|^2.
	\ean

Let $h=f+\frac{3}{2}\log f-\frac{40}{f}$, then the above inequality becomes
\ban
\Delta_h (\lambda_1+\lambda_2)\leq -0.85(\lambda_1+\lambda_2)+20 |W^{\Sigma(s)}|^2-40\frac{\nabla f\cdot\nabla(\lambda_1+\lambda_2)}{f^2}.
\ean
Integrating over $M\setminus D(a)$, we obtain
	\ban
	&&\int_{M\setminus D(a)}\Delta_h ue^{-h}dvol\\
	&&\leq-0.85\int_{M\setminus D(a)}(\lambda_1+\lambda_2)\cdot e^{-h}dvol-40\int_{M\setminus D(a)}\frac{\nabla f\cdot\nabla(\lambda_1+\lambda_2)}{f^2}\cdot e^{-h}dvol\\
&&\quad +20\int_{M\setminus D(a)}  \left(\frac{C(\lambda_1+\lambda_2)}{f}+\frac{2\nabla f\cdot\nabla(\lambda_1+\lambda_2)}{f^2}\right) e^{-h}dvol\\
&&\leq -0.8\int_{M\setminus D(a)}(\lambda_1+\lambda_2)\cdot e^{-h}dvol,
	\ean
where $a$ is almost everywhere sufficiently large.	Denote $u=\lambda_1+\lambda_2$,
	the above inequality becomes
	\ban
	\int_{M\setminus D(a)} \Delta_h u\cdot e^{-h}dvol
	\leq -0.8\int_{M\setminus D(a)}u\cdot e^{-h} dvol \leq  0.
	\ean

\end{proof}

\begin{prop}\label{other direction}
	Let $(M^5, g, f)$ be a five-dimensional shrinking gradient Ricci soliton with $R=\frac{3}{2}$. Then
	\ban
	\int_{M\setminus D(b)} \Delta_h u\cdot e^{-h}dvol\geq 0
	\ean
	for some $b\geq a$, where $u=\lambda_1+\lambda_2$, $h=f+\frac{3}{2}\log f-\frac{40}{f}$ and constant $a$ is the same as in Proposition \ref{first ineqality}.
\end{prop}

\begin{proof}
	Notice that
	\ban
	\int_{M\setminus D(b)} \Delta_h u\cdot e^{-h}dvol
	=-\int_{\Sigma(b)} \langle\nabla u, \frac{\nabla h}{|\nabla h|}\rangle \cdot e^{-h}dvol,
	\ean
	so it sufficies to prove
	\ba\label{equivalent the other direction}
	\int_{\Sigma(b)} \langle\nabla u, \frac{\nabla h}{|\nabla h|}\rangle d\sigma_{\Sigma(b)}\leq 0
	\ea
	for some $b\geq a$.
	
	\vspace{0.5cm}
	
	For this purpose, we consider the following one parameter family of diffeomorphisms,
	\begin{equation*}
		\begin{aligned}
			\begin{cases}
				&\frac{\partial F}{\partial s}=\frac{\nabla f}{|\nabla f|^2},\\
				&F(x, a)=x \in \Sigma(a).
			\end{cases}
		\end{aligned}
	\end{equation*}
	Then $\frac{\partial }{\partial s}f(F(x, s))=\langle\nabla f, \frac{\nabla f}{|\nabla f|^2} \rangle=1$,
	and the advantage of $F$ is that it maps level set of $f$ to other level set, in particular $f(F(x, s))=s$ for any $x\in \Sigma(a)$.
	
	\smallskip
	Suppose $\{x_1, x_2, x_3, x_4\}$ are local coordinate chart of $\Sigma(a)$, on $\Sigma(s)$, let  $g(s)(\frac{\partial }{\partial x_i}, \frac{\partial}{\partial x_j}):=g(\frac{\partial F}{\partial x_i}, \frac{\partial F}{\partial x_j})$, $d\sigma_{\Sigma(s)}=\sqrt{det(g_{ij})}dx$, where $dx=dx_1\wedge dx_2\wedge dx_3\wedge dx_4$.
	Next we compute the derivatives of $d\sigma_{\Sigma(s)}$.
	\begin{align*}
		&\frac{\partial}{\partial s}d\sigma_{\Sigma(s)}=\frac{\partial}{\partial s}\sqrt{det(g_{ij})}dx\\
		=&\frac{1}{2}\cdot 2 g^{ij}\langle \nabla_{\frac{\partial F}{\partial x_i}}\frac{\partial F}{\partial s}, \frac{\partial F}{\partial x_j}\rangle d\sigma_{\Sigma(s)}\\
		=& g^{ij}\langle \nabla_{\frac{\partial F}{\partial x_i}}\frac{\nabla f}{|\nabla f|^2}, \frac{\partial F}{\partial x_j}\rangle d\sigma_{\Sigma(s)}\\
		=&\frac{1}{|\nabla f|^2}g^{ij}\nabla^2 f(\frac{\partial F}{\partial x_i}, \frac{\partial F}{\partial x_j})d\sigma_{\Sigma(s)}\\
		=&\frac{1}{|\nabla f|^2} g^{ij}\left(\frac{1}{2}g(\frac{\partial F}{\partial x_i}, \frac{\partial F}{\partial x_j})-Ric(\frac{\partial F}{\partial x_i}, \frac{\partial F}{\partial x_j})\right)d\sigma_{\Sigma(s)}\\
		=&\frac{1}{|\nabla f|^2} (2-\frac{3}{2})d\sigma_{\Sigma(s)}=\frac{1}{2s}d\sigma_{\Sigma(s)}.
	\end{align*}
	Hence, it is easy to check that
	\ban
	\frac{\partial}{\partial s}\left(\frac{1}{\sqrt{s}}d\sigma_{\Sigma(s)}\right)=\left(-\frac{1}{2}s^{-\frac{3}{2}}+\frac{1}{\sqrt{s}}\frac{1}{2s}\right) d\sigma_{\Sigma(s)}=0.
	\ean
	
	Next, define a function
	\ban
	I(s)=\int_{\Sigma(s)} u\cdot \frac{1}{\sqrt{s}} d\sigma_{\Sigma(s)}.
	\ean
	and then we can compute the derivative of $I(s)$. Actually, since $u$ is Lipschitz, we obtain for almost everywhere $s>0$,
	\ban
	I'(s)&&=\frac{d}{ds}\int_{\Sigma(s)} u\cdot \frac{1}{\sqrt{s}} d\sigma_{\Sigma(s)}\\
	&&=\int_{\Sigma(s)} \langle \nabla u, \frac{\nabla f}{|\nabla f|^2}\rangle \frac{1}{\sqrt{s}}d\sigma_{\Sigma(s)}+\int_{\Sigma(s)} u\frac{\partial}{\partial s}\left(\frac{1}{\sqrt{s}}d\sigma_{\Sigma(s)}\right)\\
	&&=\int_{\Sigma(s)} \langle \nabla u, \frac{\nabla f}{|\nabla f|^2}\rangle \frac{1}{\sqrt{s}}d\sigma_{\Sigma(s)}\\
	&&=\frac{1}{s}\int_{\Sigma(s)} \langle \nabla u, \frac{\nabla h}{|\nabla h|}\rangle d\sigma_{\Sigma(s)},
	\ean
	where $\frac{\nabla f}{\nabla f}=\frac{\nabla h}{|\nabla h|}$ and $|\nabla f|=\sqrt{s}$ were used in the last equality.
	
	\smallskip
	Moreover, since $I(s)$ tends to zero as $s\rightarrow\infty$, there exists $b>a$ such that $I'(b)\leq 0$, i.e.
	\ban
	\int_{\Sigma(b)} \langle \nabla u, \frac{\nabla h}{|\nabla h|}\rangle d\sigma_{\Sigma(b)}\leq 0;
	\ean
	this finish the proof of (\ref{equivalent the other direction}).
\end{proof}

\textbf{The Proof of Theorem \ref{main}.}
Notice that Proposition \ref{first ineqality} also holds on $M\setminus D(b)$, together with Proposition \ref{other direction} implies that
\ban
\lambda_1+\lambda_2=0
\ean
on $M\setminus D(b)$ for some $b$. This implies that
\ban
\lambda_3=\lambda_4=\lambda_5\equiv \frac{1}{2}
\ean
on $M\setminus D(b)$.
Hence the function
\ban
G=tr(Ric^3)-\frac{1}{2}|Ric|^2,
\ean
is $0$ on $M\setminus D(b)$. Because $G$ is an analytic function, has to be zero, we obtain that $G\equiv 0$ on $M$. Moreover, the equation $0=\Delta_f R=R-2|Ric|^2$ implies that
\ban
G=&tr(Ric^3)-|Ric|^2+\frac{1}{4}R\\
=&\sum_{i=1}^5(\lambda_i-\frac{1}{2})^2 \lambda_i=0.
\ean
Finally we get $\lambda_1=\lambda_2\equiv 0$ and $\lambda_3=\lambda_4=\lambda_5\equiv \frac{1}{2}$ due to $Ric\geq 0$ and the continuity of $\lambda_1+\lambda_2$.

Moreover, it follows from \eqref{k1a}
\ban
K_{12}&=&\frac{\nabla_{\nabla f}R_{22} +\lambda_2(\frac{1}{2}-\lambda_2)}{f}\\
&=&\frac{\nabla f (Ric(e_2, e_2))-2 Ric(\nabla_{\nabla f}e_2, e_2)}{f}\\
&=&\frac{\nabla f (\lambda_2)-2\lambda_2g(\nabla_{\nabla f}e_2, e_2)}{f}\\
&=&0
\ean
because $\lambda_2=0$. On the other hand, $\lambda_1+\lambda_2=0$ implies $W^{\Sigma(s)}=0$ when $s\geq a$ by \eqref{intw}.
Hence,  these enable us to apply Proposition \ref{key} to obtain $\nabla Ric =0$ on $M\setminus D(a)$. Due to the analyticity of the soliton, we actually have $\nabla Ric =0$ on $M$. Then we can apply DeRham's splitting theorem to derive that $(M^5, g, f)$ is isometric to a finite quotient $\mathbb{R}^2 \times {N}^3$, where ${N}^3$ is a three-dimensional Einstein manifold, has to be $\mathbb{S}^3$.
\qed

\vspace{0.5cm}
	\textbf{Acknowledgements}.
The authors would like to thank Professor Xi-Nan Ma, Professor Huai-Dong Cao, Professor Detang Zhou,  Professor Yongjia Zhang, Professor Yu Li, Professor Xiaolong Li and Professor Mijia Lai for helpful discussions. Li's research is supported by National Natural Science Foundation of China(No. 12301062), Natural Science Foundation of Chongqing (No. CSTB2024NSCQ-MSX0537). Wu's research is supported by Natural Science Foundation of Zhejiang Province (No. LY23A010016), the Fundamental Research Funds of Zhejiang Sci-Tech University (No. 24062095-Y) and supported by the Open Research Fund of Key Laboratory of Analytical Mathematics and Applications (Fujian Normal University)(No. JAM2401). Ou's research  is partially supported by National Natural Science Foundation of China(No. 12401074, 12371061, 12371081).
\vspace{0.5cm}	

\noindent{\bf Declarations}

\noindent{\bf Conflict of interest}
The authors declare that there is no conflict of interest.
\vspace{0.5cm}	

\noindent{\bf Data availability}
No data was used for the research described in the article.

\end{document}